\newif\ifarxiv
\arxivtrue

\ifarxiv
    \documentclass[letterpaper, 11pt]{article}
    \usepackage[margin=.9in]{geometry}
    \usepackage{times}
    
    \usepackage{hyperref}       
    \usepackage{natbib}
    
    \title{Coordinate Linear Variance Reduction \\ for Generalized Linear Programming}

    \author{
    Chaobing Song \thanks{Equal contribution} \thanks{Department of Computer Sciences, University of Wisconsin-Madison. E-mail:
    \href{mailto:chaobing.song@wisc.edu}{\texttt{chaobing.song@wisc.edu}}, 
    \href{mailto:cylin@cs.wisc.edu}{\texttt{cylin@cs.wisc.edu}}, 
    \href{mailto:swright@cs.wisc.edu}{\texttt{swright@cs.wisc.edu}}, 
    \href{mailto:jelena@cs.wisc.edu}{\texttt{jelena@cs.wisc.edu}}}
    \and
    Cheuk Yin Lin \footnotemark[1] \footnotemark[2]
    \and
    Stephen J. Wright \footnotemark[2]
    \and
    Jelena Diakonikolas \footnotemark[2]
    }
    
\usepackage{amsmath,amsfonts,amsthm,bm}
\usepackage{dsfont}
\usepackage{mathtools}
\usepackage{xcolor}
\definecolor{myblue}{rgb}{0,0.08,0.45}
\definecolor{mygreen}{rgb}{0,0.65,0.66}
\usepackage[utf8]{inputenc} 
\usepackage[T1]{fontenc}    
\usepackage{url}            
\usepackage{booktabs}       
\usepackage{nicefrac}       
\usepackage{microtype}      

\usepackage{thmtools}
\usepackage{thm-restate}

\newcommand{\jd}[1]{{\color{mygreen}{\textbf{JD:} #1}}}

\usepackage{booktabs}       
\usepackage{nicefrac}       
\usepackage{microtype}      

\usepackage{algorithm,algorithmic}
\usepackage[algo2e,ruled,vlined]{algorithm2e}
\usepackage{graphicx}
\usepackage{subfig}

\newtheorem{theorem}{Theorem}
\newtheorem{lemma}{Lemma}

\newtheorem{remark}{Remark}
\newtheorem{proposition}{Proposition}
\newtheorem{assumption}{Assumption}

\newtheorem{definition}{Definition}

\def\1{\bm{1}}

\def\vzero{{\bm{0}}}
\def\vone{{\bm{1}}}
\def\vmu{{\bm{\mu}}}
\def\vxi{{\bm{\xi}}}
\def\vnu{{\bm{\nu}}}

\def\va{{\bm{a}}}
\def\vb{{\bm{b}}}
\def\vc{{\bm{c}}}

\def\vl{{\bm{l}}}

\def\vp{{\bm{p}}}
\def\vq{{\bm{q}}}
\def\vr{{\bm{r}}}
\def\vs{{\bm{s}}}
\def\vt{{\bm{t}}}
\def\vu{{\bm{u}}}
\def\vv{{\bm{v}}}
\def\vw{{\bm{w}}}
\def\vx{{\bm{x}}}
\def\vy{{\bm{y}}}
\def\vz{{\bm{z}}}

\newcommand{\dd}{\mathrm{d}}

\def\vxt{\tilde{\bm{x}}}
\def\vyt{\tilde{\bm{y}}}

\def\vwt{\tilde{\bm{w}}}
\def\vxh{\hat{\bm{x}}}
\def\vyh{\hat{\bm{y}}}

\def\vwh{\hat{\bm{w}}}

\def\mA{{\bm{A}}}

\def\mF{{\bm{F}}}

\def\mW{{\bm{W}}}

\DeclareMathAlphabet{\mathsfit}{\encodingdefault}{\sfdefault}{m}{sl}
\SetMathAlphabet{\mathsfit}{bold}{\encodingdefault}{\sfdefault}{bx}{n}

\def\gC{{\mathcal{C}}}

\def\gF{{\mathcal{F}}}

\def\gL{{\mathcal{L}}}

\def\gP{{\mathcal{P}}}

\def\gW{{\mathcal{W}}}
\def\gX{{\mathcal{X}}}

\def\sP{{\mathbb{P}}}
\def\sQ{{\mathbb{Q}}}
\def\sR{{\mathbb{R}}}

\def\sZ{{\mathbb{Z}}}

\newcommand{\E}{\mathbb{E}}

\newcommand{\R}{\mathbb{R}}

\DeclareMathOperator*{\argmin}{arg\,min}
\DeclareMathOperator*{\prox}{prox}

\DeclareMathOperator*{\dom}{dom}

\DeclareMathOperator{\st}{s.t.}

\newcommand{\innp}[1]{\left\langle #1 \right\rangle}
\usepackage[colorinlistoftodos]{todonotes}

\usepackage{psfrag}

\usepackage{multicol}
\usepackage{multirow}

\usepackage{balance}
\usepackage{threeparttable}

\newcommand{\vrpda}{\textsc{vrpda}$^2$}

\newcommand{\clvr}{\textsc{clvr}}

\else

    \documentclass{article}
    
    \PassOptionsToPackage{numbers, sort&compress}{natbib}

    \usepackage[final]{neurips2022/neurips_2022}

    \usepackage{hyperref}       
    \usepackage{url}            
    \usepackage{booktabs}       
    \usepackage{amsfonts}       
    \usepackage{nicefrac}       
    \usepackage{microtype}      
    \usepackage{xcolor}         

    \title{Coordinate Linear Variance Reduction\\ for Generalized Linear Programming}

    \author{
        Chaobing Song\thanks{Equal contribution} \\
        University of Wisconsin-Madison \\
        \texttt{chaobing.song@wisc.edu}
        \And
        Cheuk Yin Lin\footnotemark[1] \\
        University of Wisconsin-Madison \\
        \texttt{cylin@cs.wisc.edu}
        \AND
        Stephen J. Wright \\
        University of Wisconsin-Madison \\
        \texttt{swright@cs.wisc.edu}
        \And
        Jelena Diakonikolas \\
        University of Wisconsin-Madison \\
        \texttt{jelena@cs.wisc.edu}
    }

\fi

\begin{document}

\maketitle

\begin{abstract}
    We study a class of generalized linear programs (GLP) in a large-scale setting, which includes a simple, possibly nonsmooth convex regularizer and simple convex set constraints. By reformulating (GLP) as an equivalent convex-concave min-max problem, we show that the linear structure in the problem can be used to design an efficient, scalable first-order algorithm, to which we give the name \emph{Coordinate Linear Variance Reduction} (\textsc{clvr}; pronounced ``clever''). 
    \textsc{clvr} yields improved complexity results for (GLP) that depend on the max row norm of the linear constraint matrix in (GLP) rather than the spectral norm. When the regularization terms and constraints are separable, \textsc{clvr} admits an efficient lazy update strategy that makes its complexity bounds scale with the number of nonzero elements of the linear constraint matrix in (GLP) rather than the matrix dimensions. Further, for the special case of linear programs and by exploiting sharpness, we propose a restart scheme for \textsc{clvr} to obtain empirical linear convergence. Finally, we show that Distributionally Robust Optimization (DRO) problems with ambiguity sets based on both $f$-divergence and Wasserstein metrics can be reformulated as (GLPs) by introducing sparsely connected auxiliary variables. We complement our theoretical guarantees with numerical experiments that verify our algorithm's practical effectiveness in terms of wall-clock time and number of data passes.
\end{abstract}

\section{Introduction}\label{sec:intro}
We study the following generalized linear program (GLP):
\begin{equation}\tag{GLP}\label{eq:glp}
\begin{aligned}
 \min_{\vx}\; &\; \big\{ \vc^T\vx + r(\vx):  
\mA\vx = \vb, \;  \vx \in \gX\big\}, 
\end{aligned}   
\end{equation}
where $\vx,\vc\in\sR^d,$ $\mA\in\sR^{n\times d}$, $\vb\in\sR^n$, $r:\sR^{d}\rightarrow \sR$ is a convex regularizer, and $\gX\subseteq \sR^d$ is a closed convex set, such that a proximal/projection operator involving $r$ and $\gX$ can be computed efficiently.
When $\gX$ is the nonnegative orthant $\{\vx: x_i \ge 0, i\in[d]\}$ and $r\equiv 0$, \eqref{eq:glp} reduces to the standard form of a linear program (LP). 
When $\gX$ is a convex cone and  $r \equiv 0$, \eqref{eq:glp} reduces to a conic linear program.
\eqref{eq:glp} is an important paradigm in traditional engineering disciplines such as transportation, energy, telecommunications, and manufacturing. 
In modern data science, we note the renaissance of \eqref{eq:glp} due to its modeling power in such areas as reinforcement learning~\cite{de2003linear}, optimal transport~\cite{villani2009optimal}, and neural network verification~\cite{liu2020algorithms}. 
For traditional engineering disciplines with  moderate scale or exploitable sparsity, off-the-shelf interior point methods that form and factorize matrices in each iteration are often good choices as practical solvers~\cite{gurobi}. 
In data science applications, however, where the data are often dense or of extreme scale, the amount of computation and/or memory required by matrix factorization is prohibitive. 
Thus, first-order methods that avoid matrix factorizations are potentially appealing options.
In this context, because the presence of the linear equality constraint in \eqref{eq:glp} may complicate projection operations onto the feasible set, we consider an equivalent reformulation of \eqref{eq:glp} as a min-max problem involving the Lagrangian:
\begin{align}\tag{PD-GLP}\label{eq:pd-glp}
\min_{\vx\in\gX\subset\sR^d} \, \max_{\vy\in\sR^n} \, \Big\{ \gL(\vx,\vy) := \vc^T\vx + r(\vx)   + \vy^T\mA\vx  - \vy^T\vb\Big\}. 
\end{align}
In data science applications, both $n$ and $d$ can be  very large. 
\eqref{eq:pd-glp} can be viewed as a structured bilinearly coupled 
min-max 
problem,
where the linearity of $\gL(\vx,\vy)$ in the dual variable vector $\vy$ is vital to our algorithmic development. 


\subsection{Background}
While there have been few papers that directly address~\eqref{eq:pd-glp} --- some special cases have been considered in~{\cite{mangasarian1979nonlinear,mangasarian1984normal,mangasarian2004newton,gao2017first,xu2020primal,zhu2020linear,zhao2022iteration,carmon2020coordinate}} ---
there has been significant recent work on  first-order methods for general bilinearly coupled convex-concave min-max problems. 
Deterministic first-order methods include the {proximal point method} (PPM)~\cite{rockafellar1976monotone}, the extragradient/mirror-prox method (EGM)~\cite{korpelevich1976extragradient,nemirovski2004prox}, the primal-dual hybrid gradient (PDHG) method~\cite{chambolle2011first}, and the alternating direction method of multipliers (ADMM)~\cite{douglas1956numerical}. All these methods have per-iteration cost $\Theta(\mathrm{nnz}(\mA))$ and convergence rate $1/k$, where $\mathrm{nnz}(\mA)$ denotes the number of nonzero elements of $\mA$ and $k$ is the number of iterations. 

For better scalability, stochastic counterparts of these methods have been proposed. \cite{juditsky2011solving, ouyang2013stochastic,bianchi2016ergodic,patrascu2017nonasymptotic} have used ``vanilla'' stochastic gradients to replace the full gradients of their deterministic counterparts.
\cite{carmon2019variance, hamedani2020stochastic, alacaoglu2021stochastic} have exploited the finite-sum structure of the interaction term $\langle\vy, \mA\vx\rangle$ involving both primal and dual variables to perform variance reduction. 
With a separability assumption for the dual variables, \cite{alacaoglu2017smooth} and \cite{chambolle2018stochastic} have combined incremental coordinate approaches on the dual variables with an implicit variance reduction strategy on the primal variables. 
Recently, under a separability assumption for dual variables, \cite{song2021variance} proposed a new incremental coordinate method with an initialization step that requires a single access to the full data. 
This approach, known as \emph{variance reduction via  primal-dual  accelerated dual averaging} (\vrpda), obtains the first theoretical bounds that are better than their deterministic counterparts in the class of incremental coordinate approaches. 
The \vrpda~algorithm serves as the main motivation for our approach. 

It is of particular interest to design algorithms that scale with the number of nonzero elements in $\mA$ for at least two reasons: (i) the data matrix can be sparse; and (ii) when we consider simplified reformulations of certain complicated models, we often need to introduce sparsely connected auxiliary variables. Nevertheless,  the randomized coordinate algorithms of \cite{alacaoglu2017smooth, chambolle2018stochastic, song2021variance} have $O(d)$ per-iteration cost regardless of the sparsity of $\mA.$ 
To address this issue, \cite{fercoq2019coordinate,latafat2019new} have proposed incremental primal-dual coordinate methods with per-iteration cost that scales with the number of nonzero elements in the  row of $\mA$ used in each iteration, at the price of needing to take a smaller step than for dense $\mA$. Moreover, \cite{alacaoglu2020random} has proposed a random extrapolation approach that admits both low per-iteration cost and larger step size. 
Despite these developments, all these algorithms produce less accurate iterates than the methods with $O(d)$ per-iteration cost, thus degrading their worst-case complexity. 
\footnote{Subsequent to this paper, a version of the \textsc{PURE-CD} algorithm of \cite{alacaoglu2020random} that exploits sparsity in $\mA$ was developed and analyzed in \cite{alacaoglu22a}.}

Finally, for the special case of LP, based on the positive Hoffman constant~\cite{hoffman2003approximate}, \cite{applegate2021faster} proved that the 
primal-dual formulation of LP exhibits a sharpness property that lower-bounds the growth of a normalized primal-dual gap from the same work. 
Leveraging this sharpness property, \cite{applegate2021faster} proposed a restart scheme for the deterministic first-order methods discussed above to obtain linear convergence. \cite{applegate2021practical} further extended this restart strategy using various heuristics to improve practical performance.

\subsection{Motivation}

We sharpen the focus from general bilinearly coupled convex-concave min-max problems to~\eqref{eq:glp} and its primal-dual formulation~\eqref{eq:pd-glp}, because many complicated models can be reformulated as \eqref{eq:glp} and because this formulation possesses additional structure that can be exploited in algorithm design. 
Our motivation for focusing on~\eqref{eq:glp} is to bridge the large gap between the well-studied stochastic variance reduced first-order methods~\cite{johnson2013accelerating,allen2017katyusha,song2020variance,song2021variance} and the increasingly popular and complicated, yet highly structured large-scale problems arising in distributionally robust optimization (DRO)~\cite{wiesemann2014distributionally, shafieezadeh2015distributionally, namkoong2016stochastic, esfahani2018data,hu2018does, duchi2018learning, li2019first, duchi2021statistics, yu2021fast}; 
see also a recent survey by \cite{rahimian2019distributionally} 
and references therein. 

 
For DRO problems with ambiguity sets defined by $f$-divergence \cite{namkoong2016stochastic,hu2018does,levy2020large}, the original formulation is a nonbilinearly coupled convex-concave min-max problem. 
Even the well constructed reformulation in~\cite{levy2020large} does not admit unbiased stochastic gradients, leading to complicated algorithms and analysis. 
For DRO problems with ambiguity sets defined by Wasserstein metric \cite{shafieezadeh2015distributionally, esfahani2018data, li2019first,yu2021fast,HoNW20a}, the original formulation is in general infinite-dimensional. 
(Finite-dimensional reformulations  \cite{shafieezadeh2015distributionally,esfahani2018data} exist for special cases of logistic regression and smooth convex losses.) 
Solvers that have been proposed for DRO with Wasserstein metric are either multiple-loop deterministic ADMM~\cite{li2019first} or are designed for general convex-concave problems \cite{yu2021fast}.


By introducing auxiliary variables with sparse connections,\footnote{``Sparse connections'' here means that even though the newly introduced variables may substantially increase the problem dimensions, the number of nonzero entries in the constraint matrix remains of the same order. } 
we show that DRO with ambiguity sets based on both $f$-divergence and the Wasserstein metric can be reformulated as  \eqref{eq:glp}. Thus, complicated DRO problems can be addressed by a simple, efficient, and scalable algorithm for \eqref{eq:glp}. 
Our algorithm for solving \eqref{eq:glp} and the proposed reformulations of DRO are our main contributions. 
\subsection{Contributions}\label{sec:Contribution}
\paragraph{Algorithm.} Motivated by \vrpda~\cite{song2021variance}, we propose a simple, efficient, and scalable algorithm for  \eqref{eq:pd-glp}. Our algorithm combines an incremental \emph{coordinate} method with exploitation of the \emph{linear} structure for the dual variables in \eqref{eq:pd-glp} and the implicit  \emph{variance reduction} effect in the algorithm, so we name it \emph{coordinate linear variance reduction} (\clvr, pronounced ``clever''). \clvr~is inspired by \vrpda~but customized to the particular structure of \eqref{eq:pd-glp}. 
In particular, by exploiting the fact that the max problem is linear and unconstrained in the dual variable vector $\vy \in \R^n$, we find that the expensive initialization step used in \vrpda~is not needed and we can take simpler and larger steps. 
Further, in the structured case in which $\mA$ is sparse and  the convex constraint set $\gX$ and the regularizer $r(\vx)$ are fully separable\footnote{We state the results here for the fully separable setting for convenience of comparison; however, our results are also applicable to the block separable setting.}, we show that the dual averaging update in \clvr~enables us to design an efficient lazy update strategy for which  the per-iteration cost of \clvr~scales with the number of nonzero elements of the selected row from $\mA$ in each iteration, which is potentially much lower than the order-$d$ cost in \vrpda.    
Finally, \clvr~uses extrapolation on dual variables rather than on primal variables considered in \vrpda, which significantly reduces implementation complexity of our lazy update strategy for structured variants of \eqref{eq:pd-glp}. On the technical side, although both \clvr~and \vrpda~are randomized algorithms that bound the primal-dual gap in expectation, the guarantee provided by \clvr~is stronger as it allows bounding the expectation of the supremum gap as opposed to the supremum of expected gap in \vrpda.

To state our complexity results, we make the following scaling assumption. 
\begin{assumption}\label{ass:R-L}
$L:=\|\mA\|$ and each row of $\mA$ in \eqref{eq:glp} is normalized with Euclidean norm $R$. 
\end{assumption}
Preprocessing in modern LP solvers \cite{gurobi} often ensures normalized rows/columns for the data matrix.
Observe that $R \le  L \le \sqrt{n}R$, 
the upper bound being achieved when all elements of $\mA$ have identical value. 
Although the latter case is extreme, there exist ill-conditioned practical datasets where we can expect significant performance gains if the complexity can be reduced from $O(L)$ to $O(R)$. 
(We provide empirical comparison between the values of $L$ and $R$ in practical problems in Section~\ref{sec:num-experiments-discussion}.)

In Table~\ref{tb:general-result}, we give the overall complexity bounds (total number of arithmetic operations) and the per-iteration cost of a representative set of existing algorithms, including our \clvr~algorithm, for solving a structured form of \eqref{eq:pd-glp} in which the set $\gX$ and the function $r$ have separable structure: $\gX = \gX_1\times\cdots\times \gX_d$ with $\gX_i\in \sR \,(i\in[d])$ and  $r(\vx) := \sum_{i=1}^d r(x^i).$ 
To make the complexity results comparable, we assume further that for the stochastic algorithms \cite{chambolle2018stochastic, alacaoglu2021stochastic, song2021variance} and our \clvr~algorithm, we draw one row of $\mA$ per iteration uniformly at random. 
The general convex setting corresponds to $r(\vx)$ being general convex ($\sigma = 0$), while the strongly convex setting corresponds to $r(\vx)$ being $\sigma$-strongly convex ($\sigma>0$).

As shown in Table \ref{tb:general-result}, all the algorithms have optimal dependence on $\epsilon$ \cite{ouyang2018lower}, while the dependence on the ambient dimensions $n,d$, the number of nonzero elements of $\mA$ ($\mathrm{nnz}(\mA)$), and the constants $L$ and $R$ are quite different.  
For both the general convex and strongly convex settings and among coordinate-type methods, \clvr~is the first algorithm that reduces the runtime dependence on the input matrix size from $n d$ to $\mathrm{nnz}(\mA)$. 
Moreover, the complexity of \clvr~depends on the max row norm $R$ rather than the spectral norm $L$, and the per-iteration cost of \clvr~depends only on the nonzero elements of the selected row from $\mA$ in each iteration, which can be far less than $d$.



%
By exploiting the linear structure again, we provide explicit guarantees for both the objective value and the constraint satisfaction of~\eqref{eq:glp}. 
Further, the analysis of \clvr~applies to the more general {\em block-coordinate} update setting, which is better suited to modern parallel computing platforms.     
Finally, following the restart strategy based on the \emph{normalized duality gap}  for LP introduced in \cite{applegate2021faster}, we propose a more straightforward strategy to restart our \clvr~algorithm (as well as other iterative algorithms for \eqref{eq:pd-glp}): Restart the algorithm every time a widely known  metric for LP optimality~\cite{andersen2000mosek} halves. Compared with the normalized duality gap, the LPMetric can be computed more efficiently and in a more straightforward fashion.

 




\begin{table*}[t!]
\centering
\begin{threeparttable}[b]
\begin{small}
\caption{Overall complexity and per-iteration cost for solving  structured \eqref{eq:pd-glp}. 
(``---'' indicates that the corresponding result does not exist or is unknown.)}
\tabcolsep=0.1cm 
\begin{tabular}{|c|c|c||c|c|}
\hline 
\multirow{2}{*}{Algorithm}      &       {General Convex}                   & {Strongly Convex}               &     {Per-Iteration}            \\
                                &      (Primal-Dual Gap)             &     (Distance to Solution)   &                Cost            \\
\hline
\textsc{pdhg}   &    \multirow{2}{*}{$O\big( \frac{\mathrm{nnz}(\mA)  L }{\epsilon}\big)$ }         &  \multirow{2}{*}{$O\big( \frac{(\mathrm{nnz}(\mA) +n+d) L}{\sigma\sqrt{\epsilon}}\big)$}    &        \multirow{2}{*}{$O(\mathrm{nnz}(\mA))$}  \\
CP(\citeyear{chambolle2011first})
& & &\\
\hline
\textsc{spdhg}   &    \multirow{2}{*}{$O\big( \frac{n d L}{\epsilon}\big)$ }         &  \multirow{2}{*}{$O\big( \frac{n d L}{\sigma\sqrt{\epsilon}}\big)$}  &        \multirow{2}{*}{$O(d)$}    \\
CERS(\citeyear{chambolle2018stochastic})
&  & &\\
\hline
\textsc{evr}   &    \multirow{2}{*}{$O\big(\mathrm{nnz}(\mA) +  \frac{ \sqrt{\mathrm{nnz}(\mA)(n+d)n}R}{\epsilon}   \big)$ }         &   \multirow{2}{*}{---}    &        \multirow{2}{*}{$O(n+d)$}    \\
 AM (\citeyear{alacaoglu2021stochastic})  &  & &\\
\hline
{\vrpda}              &   \multirow{2}{*}{$O(nd \log {\min\{ \frac{1}{\epsilon}, n\}}
+ \frac{ndR}{\epsilon})$}               &     \multirow{2}{*}{$O(nd \log \min\{ \frac{1}{\epsilon}, n\} + \frac{nd R}{\sigma\sqrt{\epsilon}})$}     &              \multirow{2}{*}{$O(d)$ }          \\
SWD(\citeyear{song2021variance})   &              &   & \\  
\hline
\textsc{clvr}              &   \multirow{2}{*}{$O(  \frac{ \mathrm{nnz}(\mA) R }{\epsilon} )$}                &     \multirow{2}{*}{$O(\frac{\mathrm{nnz}(\mA) R}{\sigma\sqrt{\epsilon}})$}     &  \multirow{2}{*}{$O(\mathrm{nnz}(\mathrm{row}(\mA)))$}          \\
(\textbf{This Paper})  &          &   & \\  
\hline
\end{tabular}\label{tb:general-result}
\end{small}
\end{threeparttable}
\end{table*}
\paragraph{DRO reformulations.} 
When the loss function is convex, DRO problems with  ambiguity sets based on $f$-divergence \cite{namkoong2016stochastic} or Wasserstein metric \cite{esfahani2018data} are convex. However, because both problems either have complicated constraints or are infinite-dimensional, vanilla first-order methods are  inapplicable. 

For DRO with $f$-divergence, we show that by using convex conjugates and introducing auxiliary variables, the problem can be reformulated as a~\eqref{eq:glp}. 
As a result, the issue of biased stochastic gradients encountered in~\cite{levy2020large} does not arise, and  \clvr~can be applied.  
Even though the resulting problem has larger dimensions,  due to the sparseness of the introduced auxiliary variables and the lazy update strategy of \clvr, it can be solved with complexity scaling 
only with the number of nonzero elements of the data matrix.  Due to being cast as a \eqref{eq:glp}, the DRO problem can be solved with $O(1/\epsilon)$ iteration complexity with \clvr, while existing methods such as~\cite{levy2020large} have $O(1/\epsilon^2)$ iteration complexity, with higher iteration cost because of the batch of samples needed to reduce bias. 
This improvement is enabled in part by considering the primal-dual gap (rather than the primal gap considered in \cite{levy2020large}) and by allowing the constraints to be approximately satisfied (see Corollary~\ref{prog:obj-constraint}).



For DRO with Wasserstein metric, following the reformulation of~\cite[Theorem~1]{shafieezadeh2015distributionally}, we show further that the problem can be cast in the form of \eqref{eq:glp}. 
Compared with the existing reformulations \cite{shafieezadeh2015distributionally,esfahani2018data,li2019first,yu2021fast}, our reformulation can handle both smooth and nonsmooth convex loss functions. 
In fact, our reformulation can provide a more compact form for nonsmooth piecewise-linear convex loss functions (such as hinge loss). 
Moreover, compared with algorithms customized to this problem~\cite{li2019first} and extragradient methods~\cite{korpelevich1976extragradient, nemirovski2004prox, yu2021fast} for general convex-concave min-max problems, our \clvr~method attains the best-known iteration complexity and per-iteration cost, as shown in Table~\ref{tb:general-result}.

\vspace{-2mm}
\section{Notation and preliminaries}\label{sec:prelims}
For any positive integer $p$, we use $[p]$ to denote $\{1, 2, \dots, p\}$. 
We assume that there is a given 
partition of the set $[n]$ into sets $S^j$, $j \in [m],$ where $|S^{j}| = n^j > 0$ and $\sum_{j=1}^m n^{j} = n$. For $j\in[m]$, we use $\mA^{S^{j}}$ to denote the submatrix of $\mA$ with rows indexed by $S^{j}$ and $\vy^{S^{j}}$ to denote the subvector of $\vy$ indexed by $S^{j}$. 
We use $\vzero_d$ and $\vone_d$ to denote the vectors with all ones and all zeros in $d$ dimensions, respectively. 
Unless otherwise specified, we use $\|\cdot\|$ to denote the Euclidean norm for vectors and the spectral norm for matrices. 
For a given proper convex lower semi-continuous function $f:\sR\rightarrow\sR \cup \{+\infty\},$ we define the convex conjugate in the standard way as $f^*(y) = \sup_{x\in\sR}\{yx - f(x)\}$ (so that $f^{**} = f)$. 
For a vector $\vu,$ the inequality $\vu\ge \vzero$ is applied entry-wise. 
For a convex function $r(\vx)$, we use $r'(\vx)$ to denote an element of the subdifferential set $\partial r(\vx)$.
The proximal operator of $r(\vx)$ over $\gX$ is 
%
\begin{equation}\label{eq:prox}
\prox\nolimits_{r}(\hat{\vx}) =  \argmin_{\vx\in\gX}\Big\{  \frac{1}{2}\|\vx - \hat{\vx}\|^2 +  r(\vx)\Big\}.   
\end{equation}

Further, we make the following assumptions, which apply throughout the convergence analysis.  

\begin{assumption}\label{assmpt:Nash}
\eqref{eq:pd-glp} attains at least one primal-dual solution  $(\vx^*, \vy^*)$. $\gW^*$ denotes the set of all primal-dual solutions. 
\end{assumption}
Due to the convex-concave property of \eqref{eq:pd-glp}, $\gW^*$ is a convex set in $\gX\times \sR^n.$
\begin{assumption}\label{ass:L-hat}
$\hat{L} = \max_{j\in[m]}\|\mA^{S^{j}}\|$ is given at the input, where $\|\mA^{S^{j}}\| = \max_{\|\vx\|\le 1}\|\mA^{S^{j}}\vx\|.$
\end{assumption}
Note that $\hat{L}$ can be obtained either via preprocessing of the data or by parameter tuning. 
By combining Assumptions~\ref{ass:R-L} and~\ref{ass:L-hat}, it follows that $R\le \hat{L}\le \sqrt{\max_{j\in[m]}|S^{j}|}R.$
\begin{assumption}\label{assmpt:r}
$r(\vx)$ is $\sigma$-strongly convex $(\sigma\ge 0)$; that is, for all $\vx_1$ and $\vx_2$ in $\gX$ and all  $r'(\vx_2)\in\partial r(\vx_2)$, we have 
$r(\vx_1) \ge r(\vx_2) + \langle r'(\vx_2), \vx_1 - \vx_2\rangle + \frac{\sigma}{2}\|\vx_1 - \vx_2\|^2.$
\end{assumption}

For convex-concave min-max problems, a common metric for measuring solution quality is the primal-dual gap, which, for a feasible solution $(\vx, \vy)$ of  \eqref{eq:pd-glp}, 
is defined by
\begin{align}
\sup_{(\vu, \vv)\in\gX\times \sR^n}\{\gL(\vx, \vv) -  \gL(\vu, \vy)\}.       
\end{align}
However, as the domain of $\vv$ is unbounded, the primal-dual gap can be infinite, which makes it a poor metric for measuring the progress of algorithms. As a result, for measuring the progress of our algorithm, we consider the following restricted primal-dual gap instead: 
\begin{align}
\sup_{(\vu, \vv)\in\gW}\{\gL(\vx, \vv) -  \gL(\vu, \vy)\},       
\end{align}
where $\gW\subset \gX\times \sR^n$ is a compact (i.e., closed and bounded) convex set. 
The use of a restricted version of primal-dual gap is standard in the existing literature; see, e.g.,~\cite{nesterov2007dual,chambolle2011first}.

%

\section{The CLVR algorithm}
\label{sec:clvr}


\subsection{Algorithm and analysis for general formulation}\label{sec:general-algorithm}
\label{sec:alg.basic}

Algorithm~\ref{alg:clvr-impl} specifies \clvr~for~\eqref{eq:pd-glp} in the general setting. 
The algorithm alternates the full update for $\vx_k$ in Step~4 ($O(d)$ cost) with an incremental block coordinate update for $\vy_k$ in Steps 5 and 6 (with $O(|S^{j_k}|d)$ cost for dense $\mA$). 
The auxiliary variables $\vz_k$ and $\vq_k$ accumulate the cancellation terms in the estimation sequence and give a pathway to a straightforward development of the lazified CLVR, which appears as Algorithm~\ref{alg:clvr-lazy2} in the appendix.
The cost of updating auxiliary vectors $\vz_k$ and $\vq_k$ is $O(|S^{j_k}|d)$ and $O(d)$, respectively.  
In essence, \clvr~is a primal-dual coordinate method that uses a \emph{dual averaging} update for $\vx_k$, then updates the state variables $\{\vq_k\}$ by a \emph{linear recursion}, and computes $\vx_k$ from $\vq_{k-1}$ via a \emph{proximal step} without direct dependence on $\vx_{k-1}$. 
The output $\tilde{\vx}_K$ is a convex combination of the iterates $\{\vx_k\}_{k=1}^K$, as is standard for primal-dual methods. 
However, $\tilde{\vy}_K$ is only an \emph{affine} (not  convex) combination of $\{\vy_k\}_{k=0}^K$, as it involves the term $-(m-1)\vy_0$ (whose coefficient is  negative) and some of the coefficients $m a_k - (m-1)a_{k+1}$ multiplying $\vy_k$ for $k \in \{1, \dots K-1\}$ may also be negative.  
An affine combination still provides valid bounds because the dual variable vector $\vy$ appears linearly in \eqref{eq:pd-glp}. 
Moreover, in Step~9, the term $m a_k(\vz_k - \vz_{k-1})$ serves to cancel certain errors from the randomization of the update w.r.t.~$\vy_k$, thus playing a key role in implicit variance reduction. 

\begin{algorithm}[ht!]
\caption{Coordinate Linear Variance Reduction (\clvr)}\label{alg:clvr-impl}
\begin{algorithmic}[1]
\STATE \textbf{Input: } $\vx_0\in\gX, \vy_0 \in\sR^n, \vz_0 = \mA^T\vy_0$, $\gamma>0, \hat{L} >0, \sigma\ge 0,  K, m, \{S^{1}, S^{2}, \ldots, S^{m}\}.$
\STATE $a_1 = A_1 = \frac{1}{2\hat{L} m}, \vq_0 = a_1 (\vz_0 + \vc)$. 
\FOR{$k = 1,2,\ldots, K$}
\STATE $\vx_k = \prox_{\frac{1}{\gamma}A_k r}(\vx_0 - \frac{1}{\gamma}\vq_{k-1}).$ 
\STATE Pick $j_k$ uniformly at random in $[m].$
\STATE $ \vy_{k}^{S^i} = 
\begin{cases}
\vy_{k-1}^{S^i}, & i \neq j_k \\
\vy_{k-1}^{S^i} + \gamma m a_k\big(\mA^{S^i}\vx_k  - \vb^{S^i}\big), &  i = j_k\\
\end{cases} $.
\STATE $a_{k+1} = \frac{\sqrt{1 + \sigma A_{k}/\gamma}}{ 2\hat{L} m}, A_{k+1} = A_{k} + a_{k+1}.$
\STATE $\vz_k = \vz_{k-1} + \mA^{S^{j_k},T}(\vy_{k}^{S^{j_k}} - \vy_{k-1}^{S^{j_k}}).$
\STATE  $\vq_k = \vq_{k-1} + a_{k+1}\big(\vz_{k}+ \vc \big)  + ma_{k}(\vz_k - \vz_{k-1}).$ \label{line:a1l9}
\ENDFOR
\STATE \textbf{return } $\tilde{\vx}_K = \frac{1}{A_K}\sum_{k=1}^K a_k \vx_k$, $\tilde{\vy}_K = \frac{1}{A_K}\sum_{k=1}^K (a_k \vy_k + (m-1)a_k(\vy_k - \vy_{k-1}))$.
\end{algorithmic}
\end{algorithm}

Theorem~\ref{thm:clvr} provides the convergence results for Algorithm~\ref{alg:clvr-impl}. The proof is provided in Appendix~\ref{appx:proofs-general}. In the theorem (as in the algorithm), $\gamma$ is a positive parameter that can be tuned. 
\begin{restatable}{theorem}{mainthmvr}\label{thm:clvr}
Let  $\vx_k$, $\vy_k$, $k \in [K],$ be the iterates of Algorithm~\ref{alg:clvr-impl} and let  $\tilde{\vx}_k$, $\tilde{\vy}_k$ be defined by
\begin{equation} \label{eq:vxyk} 
\tilde{\vx}_k = \frac{1}{A_k}\sum_{i=1}^k a_i \vx_i, \quad
\tilde{\vy}_k =\frac{1}{A_k}\sum_{i=1}^k (a_i \vy_i + (m-1)a_i(\vy_i - \vy_{i-1})),
\end{equation}
for $k \in [K]$. Let $\gW_k\subset\gX\times\sR^n$, $k \in [K]$, be a sequence of compact convex sets such that $(\tilde\vx_k,\tilde\vy_k) \in \gW_k \subset \gW \subset \gX\times\sR^n$, where $\gW$ is also convex and compact. Then: 
\begin{align} \label{eq:ks6}
\begin{split}
 & \E\Big[\sup_{(\vu, \vv)\in\gW_k}\{\gL(\tilde{\vx}_k, \vv) - \gL(\vu, \tilde{\vy}_k)\}\Big] \\
 & \le \frac{1}{A_k}\bigg( \E\Big[\frac{\gamma}{2}\|\hat{\vu} - \vx_0\|^2 + \frac{1}{\gamma}\|\hat{\vv} - \vy_0\|^2 \Big] + \frac{\gamma}{2}\|\vx^* - \vx_0\|^2 +   \frac{1}{2\gamma}\|\vy^* - \vy_0\|^2\bigg), 
\end{split}
\end{align}
where $(\hat{\vu}, \hat{\vv}) = \arg\sup_{(\vu, \vv)\in\gW_k}\{\gL(\tilde{\vx}_k, \vv) - \gL(\vu, \tilde{\vy}_k)\}. $ 
Furthermore, 
\begin{align}
\E\Big[\frac{\gamma+ \sigma A_k}{4}\|\vx_k - \vx^*\|^2 + \frac{1}{2\gamma}\|\vy_k -\vy^*\|^2\Big] \le  \frac{\gamma}{2}\|\vx^* - \vx_0\|^2 +   \frac{1}{2\gamma}\|\vy^* - \vy_0\|^2.      \label{eq:solution-distance}
\end{align}
%
%
%
Define $K_0 = \big\lceil \frac{\sigma}{18\hat{L}m\gamma} \big\rceil.$ Then in the bounds above: 
%
%
$$
A_k \ge \max\bigg\{\frac{k}{2 \hat{L} m}, \frac{\sigma}{(6\hat{L} m)^2\gamma}\Big(k - K_0 + \max\Big\{ 3\sqrt{{2 \hat{L} m\gamma}/{\sigma}}, 1\Big\}\Big)^2\bigg\}.
$$
\end{restatable}
%
%
%
{Observe that $(\hat{\vu}, \hat{\vv})$ in the theorem statement exists because of compactness of $\gW_k$ and our assumptions on $r(\cdot)$.}
The parameter $\gamma$ can be tuned to balance the relative weights of primal and dual initial quantities $\|\vx^* - \vx_0\|$  and $\|\vy^* - \vy_0\|$ (or estimates of these quantities), which can significantly influence practical performance of the method. 


In addition to the guarantee on the variational form, due to the linear structure, we also provide explicit guarantees for both the objective and the constraints in \eqref{eq:glp}, stated in the following corollary.

\begin{restatable}{corollary}{propobjconstraint}\label{prog:obj-constraint}
In Algorithm \ref{alg:clvr-impl},  for all $k\ge 1,$ $\tilde{\vx}_k$ satisfies
	\begin{align*}
	\E[ \|\vy^*\|\cdot \|\mA\tilde{\vx}_k - \vb\|] \le\;& \frac{\gamma\|\vx^* - \vx_0\|^2+ \frac{1}{2\gamma}\|\vy^* - \vy_0\|^2  + \frac{1}{\gamma}\E[\|\vv - \vy_0\|^2] }{A_k}, \\
	|\E[ (\vc^T\tilde{\vx}_k + r(\tilde{\vx}_k)) - (\vc^T\vx^* + r(\vx^*))]|\le\;& \frac{\gamma\|\vx^* - \vx_0\|^2  + \frac{1}{2\gamma}\|\vy^* - \vy_0\|^2+ \frac{1}{\gamma}\E[\|\vv - \vy_0\|^2]}{A_k}, 
	\end{align*}
where
$
\vv = 2  \frac{ \|\vy^*\|}{\|\mA\tilde{\vx}_k - \vb\|} (\mA\tilde{\vx}_k - \vb).
$
\end{restatable}

In \clvr, we allow for arbitrary $(\vx_0, \vy_0) \in \gX\times\sR^n$. 
Nevertheless, by setting $\vy=\vzero_n$, we can obtain $\vz_0 = \vzero_d$ at no cost --- a useful strategy for large-scale problems since it avoids the (potentially expensive) single matrix-vector multiplication w.r.t.~$\mA$.  






\subsection{Lazy update for sparse and structured \eqref{eq:pd-glp}}\label{sec:lazy}

In Algorithm~\ref{alg:clvr-impl}, direct computation of the iterates $(\vx_k, \vy_k)$ and the output points $(\tilde{\vx}_k, \tilde{\vy}_k)$ can be expensive.
However, \cite{dang2015stochastic} showed that it is possible to only update the averaged vector in the coordinate block chosen for that iteration. 
This strategy requires us to record the most recent update for each coordinate block and update it only when it is selected again, which is tricky and needs to be implemented carefully.
 For sparse and block coordinate-separable instances of \eqref{eq:pd-glp},
we show that by introducing auxiliary variables that are sparsely connected, we can significantly simplify \clvr~and make its complexity scale independently of the ambient dimension $n \cdot d$, instead scaling with  $\mathrm{nnz}(\mA)$.
Due to space constraints, we defer technical details, including the lazy version of \clvr~and associated proofs,  to Appendix~\ref{appx:lazy-clvr}.

\subsection{Restart scheme}\label{sec:restart}
\vspace{-2mm}
We now propose a fixed restart strategy with a fixed number of iterations per each restart epoch and discuss an adaptive restart strategy for the special case of standard-form LP, which corresponds to \eqref{eq:glp} with $r(\vx) \equiv 0$ and $\gX = \{\vx: x_i\ge 0, i\in[d]\}$. We write
\begin{equation}\tag{LP}\label{eq:lp}
 \min_{\vx} \,  \vc^T\vx \; \st \; \mA\vx = \vb, \;  \vx \ge \bm{0}_d,
\end{equation}
and the primal-dual form
\begin{equation}\tag{PD-LP}\label{eq:pd-lp}
\min_{\vx \ge \bm{0}_d} \, \max_{\vy\in\sR^n} \, \Big\{ \gL(\vx,\vy) = \vc^T\vx  + \vy^T\mA\vx  - \vy^T\vb\Big\}. 
\end{equation}
This problem has a  sharpness property that can be used to obtain linear convergence in first-order methods \cite{applegate2021faster}. 
For convenience, in the following, we define $\vw = (\vx, \vy), \hat{\vw} = (\vxh, \vyh), \tilde{\vw} = (\vxt, \vyt)$ and $\vw^* = (\vx^*, \vy^*)$. Meanwhile, for $\gamma >0,$ we denote the weighted  norm  $\|\vw\|_{(\gamma)} := \sqrt{\gamma\|\vx-\vx^*\|_2^2 + \frac{1}{\gamma}\|\vy - \vy^*\|_2^2}.$ Further, 
we use $\gW^*$ to denote the optimal solution set of the LP and define the distance to $\gW^*$ by $\text{dist}(\vw, \gW^*)_{(\gamma)} = \min_{\vw^*\in\gW^*}\|\vw-\vw^*\|_{(\gamma)}$. When $\gamma = 1$, $\|\cdot\|_{(\gamma)}$ is the standard Euclidean norm. 
Then based on \eqref{eq:pd-lp}, we can use the following classical LPMetric\footnote{In  \eqref{eq:pd-lp}, we dualize the constraint $\mA\vx = \vb$ by $\vy^T(\mA\vx-\vb)$ instead of $\vy^T(\vb - \mA\vx)$, so in our LPMetric, there exist a sign difference for $\vy$ from the more common representation such as the one in \cite{applegate2021faster}.} to measure the progress of iterative algorithms for LP:
\begin{align}
&\text{LPMetric}(\vx, \vy) \nonumber \\ 
=\,&  \sqrt{ \|\max\{-\vx, \vzero\}\|_2^2 + \|\mA\vx - \vb\|_2^2 + \|\max\{-\mA^T\vy -\vc, \vzero\}\|_2^2 + |\max\{\vc^T\vx + \vb^T\vy, 0\}|^2}, \label{eq:lpmetric} 
\end{align}
which can be explicitly and directly computed. 
For the Euclidean case ($\gamma = 1$), it is well-known~\cite{hoffman2003approximate} that there exists a Hoffman constant $H_{1}$ such that 
\begin{align}
\text{LPMetric}(\vw) \ge H_{1}\text{dist}(\vw, \gW^*)_{(1)}.          
\end{align}
Using the equivalence of norms in finite dimensions, for general $\gamma >0,$ we can conclude that there exists another constant $H_{\gamma}$ (to which we refer as the generalized Hoffman's constant) such that
\begin{align}
\text{LPMetric}(\vw) \ge H_{\gamma}\text{dist}(\vw, \gW^*)_{(\gamma)}.  \label{eq:LP-Metric}
\end{align}

Using Eq.~\eqref{eq:LP-Metric} and  Theorem~\ref{thm:clvr}, we  then obtain the following bounds for distance and LPMetric.  
\begin{restatable}{theorem}{propLPMetric}\label{thm:LPMetric}
Consider the \clvr~algorithm applied to the standard-form LP problem \eqref{eq:pd-lp}, with input $\vw_0$ and output $\vwt_k$. 
Given $\gamma>0$, define $\vw^* = \arg\min_{w \in \gW^*} \, \|\vw_0-\vw\|_{(\gamma)}$,  and define  $C_0 = \gamma + 1/\gamma +    (\sqrt{2} + 1) \|\vw_0 - \vw^*\|_{(\gamma)} + \|\vw^*\|_{(\gamma)}.$ 
Then for $H_{\gamma}$ defined as in \eqref{eq:LP-Metric}, we have 
\begin{align*}
\E\Big[\sqrt{\emph{dist}(\tilde{\vw}_k, \gW^*)_{(\gamma)}}\;\Big] \le\;&  5\sqrt{\frac{\hat{L}m C_0}{H_{\gamma}k}}\sqrt{\emph{dist}(\vw_0, \gW^*)_{(\gamma)}}, \\
\E\big[\sqrt{\emph{LPMetric}(\tilde{\vw}_k)}\big] \le\;&  5\sqrt{\frac{\hat{L}m C_0}{H_{\gamma}k}}\sqrt{\emph{LPMetric}(\vw_0)}.   
\end{align*}
\end{restatable}
As a result, by Theorem \ref{thm:LPMetric}, if we know the values of $\hat{L}, \|\vw^*\|_{(\gamma)}$ and $H_{\gamma}$, then by setting $k=\frac{100\hat{L}m C_0}{H_{\gamma}k}$, we can halve the square root of the distance and the LPMetric in expectation. Thus we can obtain linear convergence if we restart the CLVR algorithm after a fixed number of iterations. However, the values of $\|\vw^*\|_{(\gamma)}$ and $H_{\gamma}$ are often unknown and thus make this strategy unrealistic in practice.

Compared with the above fixed restart strategy, a natural strategy is to restart whenever the LPMetric halves (summarized in Algorithm~\ref{alg:restarted-CLVR} in the appendix). Since LPMetric is easy to monitor and update, implementation of this strategy is straightforward. However, bounding the number of iterations required to halve the metric (in expectation or with high probability) seems nontrivial. What can be said (based on Theorem~\ref{thm:LPMetric} and denoting by $K$ the number of iterations on \clvr~between restarts)  is that $\sP[K > \frac{50 \hat{L}m C_0}{\delta^2 H_{\gamma}}] \leq \delta.$ This follows by Markov inequality, as $\sP[K > k] = \sP\Big[\sqrt{\mathrm{LPMetric}(\tilde{\vw}_k)} > \sqrt{\frac{\mathrm{LPMetric}(\vw_0)}{2}}\Big] \leq 5\sqrt{2\frac{\hat{L}m C_0}{H_\gamma k}}.$ We provide a comparison between the adaptive restart scheme proposed in~\citep{applegate2021faster} and our proposed adaptive restart scheme in Section~\ref{subsec-restart_schemes_comparisons} to demonstrate its practical competitiveness. Although we use adaptive restart in our experiments, we defer its convergence analysis to future work. Finally, as an independent and parallel work to ours, \cite{lu2021nearly} proposed a high probability guarantee for scheduled restart for stochastic extragradient-type methods.


\section{Application: DRO}\label{sec:dro}

Consider sample vectors $\{\va_1, \va_2, \ldots, \va_n\}$ with labels $\{b_1, b_2, \ldots, b_n\}$, where $b_i\in\{1, -1\}\,(i\in[n])$. The DRO problem with $f$-divergence based ambiguity set is
\begin{align}
\min_{\vx\in\gX}\sup_{\vp\in \gP_{\rho, n}} \sum_{i=1}^n p_i g(b_i\va_i^T\vx),    \label{eq:dro-f}
\end{align}
where  $\gP_{\rho,n} = \big\{\vp\in\sR^n: \sum_{i=1}^n p_i = 1, p_i\ge 0\, (i\in[n]), D_f(\vp\|\mathbf{1}/n) \le \frac{\rho}{n}\big\}$ is the ambiguity set, $g$ is a convex loss function and $D_f$ is an $f$-divergence defined by $D_f(\vp\|\vq) = \sum_{i=1}q_if(p_i/q_i)$ with $\vp, \vq\in \big\{\vp\in\sR^n: \sum_{i=1}^n p_i = 1, p_i\ge 0\big\}$ and $f$ being a convex function~\cite{namkoong2016stochastic}.  
The formulation \eqref{eq:dro-f} is a nonbilinearly coupled convex-concave min-max problem with constraint set $\gP_{\rho,n}$ for which efficient projections are not available in general. 
When $g$ is a nonsmooth loss (e.g., the hinge loss), many well-known methods such as the  extragradient~\cite{korpelevich1976extragradient,nemirovski2004prox} cannot be used even if we could project onto $\gP_{\rho,n}$ efficiently. 
However, by introducing auxiliary variables, additional linear constraints, and simple convex constraints, we can make the interacting term between primal and dual variables bilinear, as shown next.
(See Appendix~\ref{appx:dro-proof} for a proof.)
%
\begin{restatable}{theorem}{thmfdivergence}\label{thm:f}
Let $\gX$ be a compact convex set. Then the DRO problem in Eq.~\eqref{eq:dro-f} is equivalent to 
\begin{equation} \label{eq:thm:f}
\begin{aligned}
\min_{\vx, \vu,  \vv, \vw, \vmu, \vq, \gamma} \; &\; \Big\{\gamma  + \frac{\rho\mu_1}{n} + \frac{1}{n}\sum_{i=1}^n  \mu_i f^*\Big(\frac{q_i}{\mu_i}\Big)  \Big\}&&  \\
\st\;& \;   \vw + \vv - \frac{\vq}{n} - \gamma\mathbf{1}_n = \vzero_n, && \\
\;& u_i =  b_i\va_i^T\vx,  && i\in[n]                    \\
&\;   \mu_1 = \mu_2=\cdots = \mu_n ,     &&              \\
&\;      g(u_i) \le w_i,  && i\in[n]                           \\
&\;    q_i\in \mu_i \dom(f^*),  && i\in[n]              \\
&\;     v_i\ge 0, \; \mu_i\ge 0,        &&            i\in[n]     \\
&\;      \vx\in\gX.
\end{aligned}\nonumber
\end{equation}
\end{restatable}
%

\vspace{-2mm}

In Theorem~\ref{thm:f}, the domain of the one-dimensional convex function $f^*(\cdot)$ is  an interval such as $[a, b]$, so that $q_i\in \mu_i \dom(f^*)$ denotes the inequality  $\mu_i a \le q_i \le \mu_i b$. 
Since the perspective function $\mu f^*\big(\frac{q}{\mu}\big)$ is a simple convex function of two variables, we can assume that the proximal operator for this function on the domain $\{(\mu, q):  q\in \mu \dom(f^*), \mu >0\}$ can be computed efficiently~\cite{boyd2004convex}.  
Similarly, we can assume that the constraint $g(u)\le w$ admits an efficiently computable projection operator. 
As a result, the formulation \eqref{eq:dro-f} can be solved by \clvr.
When expressing \eqref{eq:dro-f} in the form of  \eqref{eq:pd-glp}, the primal and dual variable vectors have dimensions $d+1 + 4n$ and $3n-1$, respectively. 
However, according to Table~\ref{tb:general-result}, provided that $\gX$ is coordinate separable, the overall complexity of \clvr~will only be $O\big( \frac{(\mathrm{nnz}(\mA) + n)(R+1)}{\epsilon} \big)$. 

The original DRO problem with Wasserstein metric based ambiguity set is an \emph{infinite}-dimensional \emph{nonbilinearly} coupled convex-concave min-max problem defined by 
%
\begin{equation} 
\min_{\vw\in\sR^d}\sup_{\sP\in\gP_{\rho,\kappa}} \E^{\sP}[g(b\va^T \vw)],  \label{eq:wass}
\end{equation}
%
where $\va\in\sR^d, b\in\{1,-1\},$ $\sP$ is a distribution on $\sR^d\times \{1, -1\}$, $g$ is a convex loss function and $\gP_{\rho,\kappa}$ is the Wasserstein metric-based ambiguity set  \cite{shafieezadeh2015distributionally}. 
Our reformulation for Eq.~\eqref{eq:wass} is in Appendix~\ref{app:dro-wass}.

\section{Numerical experiments} \label{sec:num-experiments-discussion}

We provide experimental evaluations of our algorithm for the reformulation of the  DRO with Wasserstein metric based on the $\ell_1$-norm (with $\kappa = 0.1$ and  $\rho = 10$) and hinge loss. 
For its LP formulation (see Theorem~\ref{thm:dro-wass} in the Appendix), we compare our \clvr~method with three representative methods: \textsc{pdhg}~\cite{chambolle2011first}, \textsc{spdhg}~\cite{chambolle2018stochastic} and \textsc{pure-cd}~\cite{alacaoglu2020random}. For all  algorithms we use LPMetric~\eqref{eq:lpmetric} as the performance measure and use a restart strategy based on successive halving of LPMetric (Section~\ref{sec:restart}) to obtain linear convergence. 
We implemented \textsc{clvr} and other algorithms in \href{https://julialang.org}{Julia}, optimizing all implementations to the extent possible. 
Full details of the experimental setup can be found in Appendix~\ref{appx:experiments-details}. 
Our code is available at~\url{https://github.com/ericlincc/Efficient-GLP}.

\paragraph{Comparison between values of $L$ and $R$.}

As described in Section~\ref{sec:intro}, a major advantage of \clvr~is that the complexity of \clvr~depends on the max row norm $R$ instead of the spectral norm $L$, which in the worst case for ill-conditioned problems can lead to a factor of $\sqrt{n}$ improvement. 
In practical problems where the problem instances are highly structured (e.g., reformulated DRO problems), $R$ can be much smaller than $L$. Table~\ref{table:L-values} provides empirical evidence for this claim. 
In all our experiments, we normalize each rows of $\mA$ to $R = 1$ as stated in Assumption~\ref{ass:R-L}, so the values of $L$ demonstrate the theoretical improvements for the experiments described in Section~\ref{sec:num-experiments-discussion}.

\begin{table*}[ht!] 
\centering
\begin{threeparttable}[b]
\begin{small}
\caption{Values of the spectral norm $L$ in the reformulated DRO problems with Wasserstein metric after each row is normalized to $R=1$.}
\tabcolsep=0.1cm 
\begin{tabular}{|c|c|c|c|}
\hline
  \text{Reformulated a9a} 	& \text{Reformulated gisette}	& \text{Reformulated rcv1} 	& \text{Reformulated news20}  \\ 
    $d = 130738, n = 97929$	&  $d = 44002, n = 28000$	& $d = 269914, n = 155198$	& $d = 5500750, n = 2770370$\\ \hline
  117.3 & 65.9 & 196.4 & 1041.6  \\ \hline
  \end{tabular}\label{table:L-values}
  \end{small}
\end{threeparttable}
\end{table*}

\paragraph{Comparison with primal-dual algorithms.}

\begin{figure*}[ht!]
\captionsetup[subfigure]{labelformat=empty}
    \centering
    \subfloat{{\includegraphics[width=0.31\textwidth]{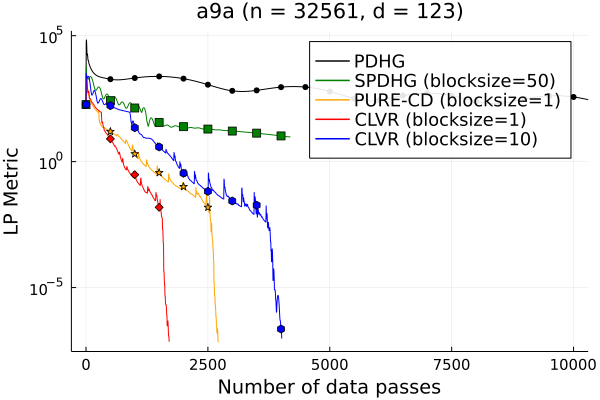} }\label{fig:a9a-datapass}}%
    \hspace{\fill}
    \subfloat{{\includegraphics[width=0.31\textwidth]{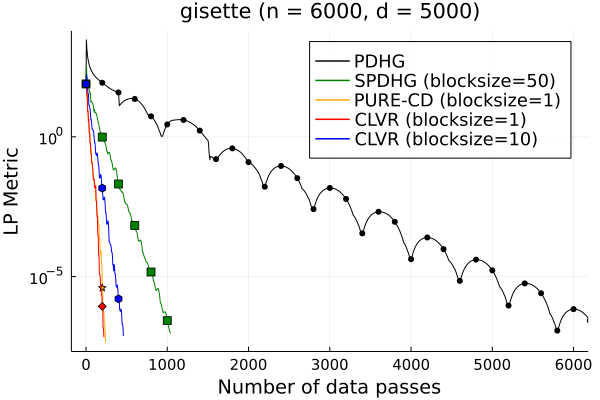} }\label{fig:gisette-datapass}}%
    \hspace*{\fill}
    \subfloat{{\includegraphics[width=0.31\textwidth]{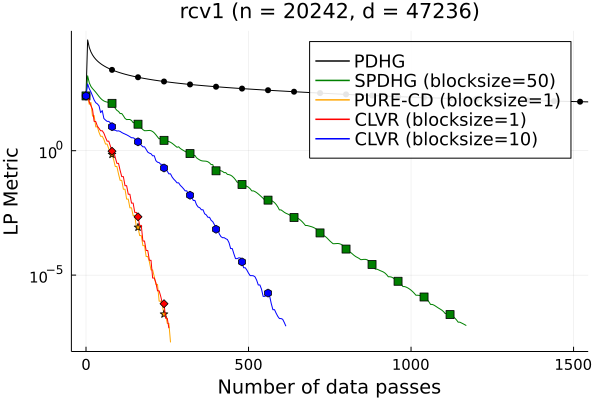} }\label{fig:rcv1-datapass}}%
    
    \hspace*{\fill}
    
    \subfloat{\includegraphics[width=0.31\textwidth]{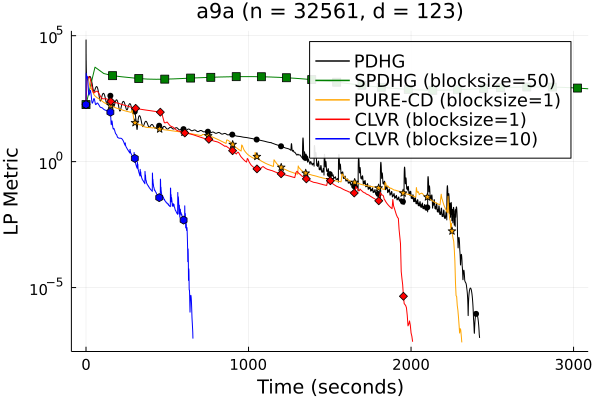}\label{fig:a9a-time}}%
    \hspace*{\fill}
    \subfloat{{\includegraphics[width=0.31\textwidth]{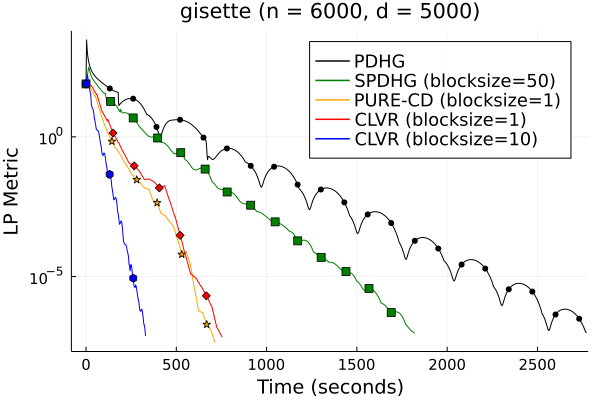}}\label{fig:gisette-time}}
    \hspace*{\fill}
    \subfloat{{\includegraphics[width=0.31\textwidth]{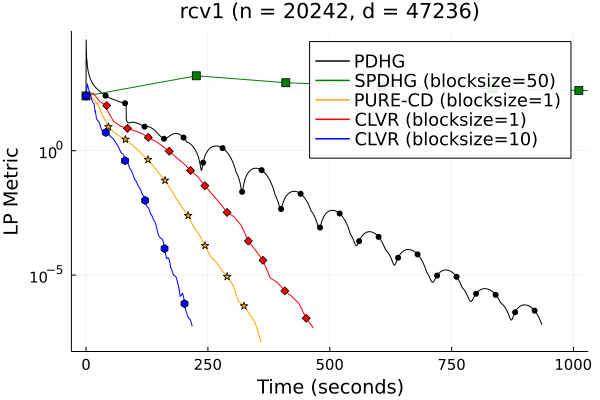} }\label{fig:rcv1-time}}%
    \caption{Comparison of numerical results in terms of number of data passes and wall-clock time.}%
    \label{fig:plots}
\end{figure*}

Figure~\ref{fig:plots} provides a comparison between algorithms in terms of the number of data passes and wall-clock time. The spikes in all the plots are due  to restarts: At the beginning of each restart cycle, the value of LPMetric increases significantly, then decreases rapidly. 
For the number of data passes (top row), \clvr~with block size $1$ and \textsc{pure-cd} perform best on all three datasets, \clvr~with block size 10 and \textsc{spdhg} with block size 50 have second-tier performance, and \textsc{pdhg} is worst. 
For the \clvr~algorithm, smaller block size corresponds to smaller $\hat{L}$ in  Assumption~\ref{ass:L-hat}, which corresponds to better complexity in terms of data passes by Theorem~\ref{thm:clvr}. 
Nevertheless, the gap between empirical performance and theoretical guarantee for \textsc{spdhg} and \textsc{pure-cd} deserves further research because, to date, they have only been shown to have the same iteration complexity as \textsc{pdhg}. 
\footnote{The later paper \cite{alacaoglu22a} describes complexity results for a newly developed version of \textsc{PURE-CD} that exploits sparsity in $\mA$.}
Empirically, on \texttt{a9a}, \clvr~with block size $1$ performs better than \textsc{pure-cd} in terms of data passes. 


In terms of  wall-clock time (bottom row of Figure~\ref{fig:plots}), because of different per-iteration costs of each algorithm and instruction-level parallelism in modern processors~\cite{hennessy2011computer}, the plots  differ significantly from the plots for  number of data passes. 
Even with block size $50$, \textsc{spdhg} spends the most wall-clock time for one data pass and is the slowest on sparse datasets \texttt{a9a} and  \texttt{rcv1}, but is faster than \textsc{pdhg} on the dense dataset \texttt{gisette}. Meanwhile, while \clvr~with block size $10$ is not best in terms of data passes, it remains fastest in terms of wall-clock time on all datasets due to cheaper per-iteration cost and instruction-level parallelism. 
On \texttt{rcv1}, the per-iteration cost of \textsc{pure-cd} is about $60\%$ of that of \clvr~with block size $1$. Hence, despite having similar performance in terms of data passes, \textsc{pure-cd} is faster than \clvr~with block size $1$, but is still slower than \clvr~with block size $10$.

\paragraph{Comparison with production linear programming solvers.}

Table~\ref{table:solvers} shows that \clvr~is competitive against production-quality linear programming solvers such as GLPK~\cite{glpk} and Gurobi~\cite{gurobi}. 
We observe that \clvr~reached accurate solutions significantly faster than GLPK and Gurobi in the reformulated problems with \texttt{gisette} and \texttt{rcv1} datasets. Although \clvr~is much slower than Gurobi(barrier) on \texttt{a9a} dataset, we believe that much of the performance gap in this case is due to the redundancy in the problem formulation with the \texttt{a9a} dataset, much of which is removed by Gurobi presolver\footnote{In our DRO instance with \texttt{a9a} dataset, Gurobi presolver removed 25\% of the columns and 58\% of the nonzeros.}. We leave presolving and other heuristic speedups of \clvr~for future work. 

\begin{table*}[ht!] 
\centering
\begin{threeparttable}[b]
\begin{small}
\caption{Comparison of numerical results between \clvr~and three production solvers for linear programming, showing time required (in seconds) for each solver to reach accuracy $10^{-8}$.}
\tabcolsep=0.1cm 
\begin{tabular}{|c|c|c|c|}
\hline
  \text{Time (seconds)}	& \text{Reformulated a9a} 	& \text{Reformulated gisette}	& \text{Reformulated rcv1} \\ 
      			& $d = 130738, n = 97929$	&  $d = 44002, n = 28000$	& $d = 269914, n = 155198$	\\ \hline
  \text{JuMP+GLPK} 	& $899$ 	& $>4\times10^4$ 	& $>4\times10^4$	\\ \hline
  \text{JuMP+Gurobi(simplex)}& $893$ 	& $2482$ 	& $7008$	\\ \hline
  \text{JuMP+Gurobi(barrier)}& $\mathbf{26}$ 	& $1039.7$ 	& $1039.5$	\\ \hline
  \text{CLVR} 		& $962$ 			& $\mathbf{697}$ 	& $\mathbf{582}$ \\
  \hline
  \end{tabular}\label{table:solvers}
  \end{small}
\end{threeparttable}
\end{table*}

\paragraph{Conclusion.}

Our preliminary numerical experiments show that \clvr~is fastest in both the number of data passes and wall-clock time on considered datasets, among all primal-dual algorithms that we implemented.
It is also competitive with production-quality linear programming solvers. 
Since it has a theoretical guarantee that matches or improves the state of the art among primal-dual methods, we believe that \clvr~could be a method of choice.

\ifarxiv
    \section*{Acknowledgments and Disclosure of Funding}
    
    CS was supported in part by the NSF grant 2023239. 
    JD and CYL acknowledge support from the NSF award 2007757. 
    JD was also supported by the Office of Naval Research under contract number N00014-22-1-2348 and the Wisconsin Alumni Research Foundation. 
    SW was supported by NSF grants 2023239 and 2224213, the DOE under subcontract 8F-30039 from Argonne National Laboratory, and the AFOSR under subcontract UTA20-001224 from UT-Austin.
    Part of this work was done while JD, CS, and SW were visiting the Simons Institute for the Theory of Computing.
\fi


\bibliography{ref.bib}
\bibliographystyle{abbrvnat}


\ifarxiv
\fi


\clearpage
\appendix


{\centering{\LARGE\bfseries Coordinate Linear Variance Reduction\\[5pt]  for Generalized Linear Programming}

  \centering{{\LARGE\bfseries Appendix}}

}

\paragraph{Outline.} The appendix of this paper is organized as follows:
\begin{itemize}
    \item Section~\ref{appx:lazy-clvr} contains details and the pseudo code for \clvr~with lazy update. 
    \item Section~\ref{appx:proofs-general} provides the proofs for Section~\ref{sec:clvr}. 
    \item Section~\ref{appx:dro-proof} provides the proofs for Section~\ref{sec:dro}. 
    \item Section~\ref{appx:experiments-details}  provides the complementary details for Section \ref{sec:num-experiments-discussion}. 
\end{itemize}

\section{Lazy update for sparse and structured instances of \eqref{eq:pd-glp}} \label{appx:lazy-clvr}

In Algorithm~\ref{alg:clvr-impl}, for dense $\mA$, the $O(|S^{j_k}|d)$ cost of Steps~6 and 8 dominates the $O(d)$ cost of Steps~4 and 9. 
However, when $\mA$ is sparse and $|S^{j_k}|$ is small, the cost of Steps 6 and 8  will be $O(\mathrm{nnz}(\mA^{S^{j_k}}))$, which may be less than the $O(d)$ cost of Steps 4 and 9. 
Using this observation, we show that the nature of the dual averaging update enables us to propose an efficient implementation whose complexity depends on $\mathrm{nnz}(\mA)$ rather than $n \cdot d$.



Recall that we partition $[n]$ into subsets $\{S^1, S^2, \ldots, S^m\}$ and use $\mA^{S^j} (j\in[m])$ to denote the $j^\mathrm{th}$ row block of $\mA.$ 
For each block $\mA^{S^j}$, we use $C^j \subset [d]$ to denote the indices of those columns of $\mA^{S^j}$ that contain at least one nonzero element. 
(Of course, $\{C^1, \dotsc,C^m\}$ is not in general a partition of $[d]$ as different subsets $S^j$ may have nonzeros in the same columns.)
We assume further that $\gX$ and $r$ are coordinate separable, that is, $\gX = \gX_1\times\cdots \times \gX_d$ with $\gX_i\subset \sR$ for all $i$ and  $r(\vx) := \sum_{i=1}^d r(\vx^i)$ with $\vx^i\in\gX_i.$   

In Step~4 of Algorithm~\ref{alg:clvr-impl}, the separability of $\gX$ and $r$ means that an update to one coordinate block of $\vx$ --- say the $C^{j_k}$ block in the update from $\vx_{k-1}$ to $\vx_k$, which requires a projection and an application of the proximal operator --- does not influence other coordinates $\vx_k^i$ for $i \notin C^{j_k}$. 
Moreover, we can efficiently maintain  an implicit representation of $\vq_k$, in terms of a newly introduced auxiliary vector $\vr_k$; see Lemma~\ref{lem:qk} of Appendix \ref{app:lazy}. Similarly, to update $\tilde{\vy}_k$ efficiently, we maintain an implicit representation of  $\tilde{\vy}_k$ via an auxiliary vector $\vs_k$, as shown in  Lemma~\ref{lem:yk} of Appendix \ref{app:lazy}.

It is generally not possible to maintain $\vxt_k$ efficiently since, in principle, all components of $\vx_k$ can change on every iteration, and we wish to avoid the $O(d)$ cost of evaluating every full $\vx_k$. 
We seek instead to output a proxy $\vxh_K$ for $\vxt_K$ from Algorithm~\ref{alg:clvr-impl} such that $\E[\vxh_K]=\vxt_K$. 
One possible choice is to  pick an index $k' \in [K]$ from the weighted discrete distribution $(\frac{a_1}{A_K}, \frac{a_2}{A_K},\ldots, \frac{a_K}{A_K})$ (computing the scalar quantities $a_k$ and $A_k$ for  $k\in[K]$ in advance), then setting $\hat{\vx}_K = \vx_{k'}$. 
A slightly more sophisticated strategy is to sample a predetermined number $\widehat{K}$ of vectors $\vx_k$, $k \in [K]$, and define $\hat{\vx}_K$ to be the simple average of these vectors. Once again, the indices are chosen from the weighted discrete distribution $(\frac{a_1}{A_K}, \frac{a_2}{A_K},\ldots, \frac{a_K}{A_K})$. 
Note that the total expected cost of the $K$ iterations of the algorithm (excluding the full-vector updates) is $O(K \mathrm{nnz}(\mA)/m)$, while the total cost of evaluating the $\widehat{K}$ full vectors $\vx_k$ and accumulating them into $\vxh_K$ is $O(\widehat{K}d)$.
Thus, for the full-step iterations not to dominate the total cost, we can choose $\widehat{K}$ to be $O(K\mathrm{nnz}(\mA)/(md))$. 
Finally, we note that Theorem~\ref{thm:clvr} is for $\vxt_k$, not  $\vxh_K$. However,  since $\E[\vxh_k]=\vxt_k$,  we  expect the convergence rates from the theorem to hold for $\vxh_K$ in practice.


An implementation of Algorithm~\ref{alg:clvr-impl} that exploits the form of $\vq_k$ in   Lemma~\ref{lem:qk} and $\vyt_k$ in Lemma \ref{lem:yk}, and evaluates explicitly only those components of $\vx_k$ needed to perform the rest of the iteration is given as Algorithm~\ref{alg:clvr-lazy2}. 
This version also incorporates the strategy for obtaining $\vxh_K$ by sampling $\widehat{K}$ iterates on which to evaluate the full vector $\vx_k$.


Due to the efficient implementation in Algorithm \ref{alg:clvr-lazy2}, to attain an $\epsilon$-accurate solution in terms of the primal-dual gap in Theorem \ref{thm:clvr}, we need $O(\frac{\text{nnz}(\mA)\hat{L}}{\epsilon})$ FLOPS, which  corresponds to $O(\frac{\hat{L}}{\epsilon})$ data passes. As a result, because a smaller batch size leads to a smaller $\hat{L}$, we attain the best performance in terms of number of data passes when the batch size is set to one. However, as modern computer architecture has particular design for vectorized operations, lower runtime is obtained
for a small batch size strictly larger than one (see Section~\ref{sec:num-experiments-discussion}).        

\begin{remark}
While we consider the case of  fully coordinate separable $r$ and $\gX$ for simplicity, our lazy update approach is also applicable to the coordinate block partitioning case in which $\gX = \gX_1\times\cdots\gX_m$ with $\gX_i\in \sR^{d_i} \,(i\in[m], \sum_{i=1}^m d_i=d)$ and $r(\vx) := \sum_{i=1}^m r(\vx^i)$ with $\vx^i\in\gX_i.$
The difference is that for each coordinate block $\mA^{S^j} (j\in[m])$, we overload $C^j \subset [m]$ to denote the set of blocks in $\mA^{S^j}$ where each coordinate block contains at least one nonzero element.
\end{remark}


\begin{remark}
Dual averaging has been shown to have significant advantage in producing sparser iterates than mirror descent in the context of online learning~\cite{xiao2010dual,lee2012manifold}. It further 
leads to better bounds in well-conditioned finite-sum optimization~\cite{song2020variance}. 
In this work, we show that dual averaging offers better flexibility with sparse matrices than does mirror descent.
\end{remark}

\begin{algorithm}[ht!]
\caption{Coordinate Linear Variance Reduction with Lazy Update (Lazy CLVR)}\label{alg:clvr-lazy2}
\begin{algorithmic}[1]
\STATE \textbf{Input: } $\vxt_0 = \vx_0=\vx_{-1}\in\gX, \vy_0 \in \sR ^n,  \vz_0 = \mA^T\vy_0,$ $ \gamma>0$, $\hat{L} >0$,  $K$, $\widehat{K}$, $m$, $\{S^{1}, S^{2}, \ldots, S^{m}\}$, $\{C^{1}, C^{2}, \ldots, C^{m}\}$.
\STATE $a_0 = A_0 = 0, a_1 = A_1 = \frac{1}{2\hat{L} m},  \vq_0 =  a_1 (\vz_0 + \vc)$, $\vr_0 = \vzero_d$, $\vs_0 = \vzero_n$. 
\FOR{$k = 1,2,\ldots, K-1$}
\STATE $a_{k+1} = \frac{\sqrt{1 + \sigma A_{k}/\gamma}}{2\hat{L} m}$, $A_{k+1} = A_{k} + a_{k+1}.$
\ENDFOR
\STATE Choose indices $\{\ell_1, \dotsc,\ell_{\widehat{K}}\}$ i.i.d. from $[K]$ according to the distribution $\left\{\frac{a_1}{A_K}, \frac{a_2}{A_K}, \ldots, \frac{a_K}{A_K}\right\}$.
\FOR{$k = 1,2,\ldots, K$}
\STATE Pick $j_k$ uniformly at random in $[m].$
\IF{$k=\ell_i$ for some $i=1,2,\dotsc,\widehat{K}$}
\STATE $\vq_{k-1} = A_k (\vc + \vz_{k-1}) + \vr_{k-1}.$
\STATE $\vx_{k} =  \prox_{\frac{1}{\gamma}A_{k} r}(\vx_0 - \frac{1}{\gamma}\vq_{k-1}).$
\ELSE
\STATE $\vq_{k-1}^{C^{j_k}} = A_k (\vc^{C^{j_k}} + \vz_{k-1}^{C^{j_k}}) + \vr_{k-1}^{C^{j_k}}.$
\STATE $\vx_{k}^{C^{j_k}} =  \prox_{\frac{1}{\gamma}A_{k} r}(\vx_0^{C^{j_{k}}} - \frac{1}{\gamma}\vq_{k-1}^{C^{j_{k}}}).$
\ENDIF
\STATE $\vy_{k}^{S^{j_k}} = \vy_{k-1}^{S^{j_k}} + \gamma m a_k\big(\mA^{S^{j_k}, C^{{j_k}}}\vx_{k}^{C^{{j_k}}} - \vb^{S^{j_k}}\big), \vy_{k-1}^{S^i}$ for all $i\neq j_k.$
\STATE $\vz_k^{C^{j_k}} = \vz_{k-1}^{C^{j_k}} + (\mA^{S^{j_k}, {C^{j_k}}})^T(\vy_{k}^{S^{j_k}} - \vy_{k-1}^{S^{j_k}})$, $\vz_k^i = \vz_{k-1}^i$ for all $i \notin C^{j_k}$;
\STATE $\vr_k^{C^{j_k}} = \vr_{k-1}^{C^{j_k}} + (ma_k - A_k)(\vz_k^{C^{j_k}}-\vz_{k-1}^{C^{j_k}})$, $\vr_k^i = \vr_{k-1}^i$ for all $i \notin C^{j_k}$;
\STATE $\vs_k^{S^{j_k}} = \vs_{k-1}^{S^{j_k}} + ((m-1)a_k - A_{k-1})(\vy_k^{S^{j_k}}-\vy_{k-1}^{S^{j_k}})$, $\vs_k^i = \vs_{k-1}^i$ for all $i \notin S^{j_k}$;
\ENDFOR
\STATE $\hat{\vx}_K = \frac{1}{\widehat{K}}\sum_{i=1}^{\widehat{K}} \vx_{\ell_i}.$
\STATE $\vyt_K = \vy_K + \frac{1}{A_K}\vs_K.$
\STATE  \textbf{return } $\hat{\vx}_K$ and $\tilde{\vy}_K$.
\end{algorithmic}
\end{algorithm}

\section{Omitted proofs from Section~\ref{sec:clvr}} \label{appx:proofs-general}

\subsection{Omitted proofs from Section \ref{sec:general-algorithm}}

We state a version of \clvr~in Algorithm~\ref{alg:clvr-v2} that is convenient for the analysis, and is equivalent to Algorithm~\ref{alg:clvr-impl} in the main body of the paper. 
Before analyzing the convergence of \clvr, we justify our claim of equivalence in Proposition~\ref{prop:equivalence-of-algs}.

\begin{proposition}\label{prop:equivalence-of-algs}
The iterates of Algorithm~\ref{alg:clvr-impl} and \ref{alg:clvr-v2} are equivalent.
\end{proposition}
\begin{proof}
To argue equivalence, we show that the iterates of Algorithm~\ref{alg:clvr-impl} and \ref{alg:clvr-v2} solve the same optimization problems. To avoid ambiguity, here we will use $\vxh_k, \vyh_k$ to denote the iterates $\vx_k, \vy_k$ in Algorithm~\ref{alg:clvr-v2}, while we retain the notation $\vx_k, \vy_k$ for the iterates of Algorithm~\ref{alg:clvr-impl}. 

Let us first start by writing an equivalent definition of $\vx_k$ in Algorithm~\ref{alg:clvr-impl}. To do so, we first unroll the recursive definitions of $\vz_k$ and $\vq_k$. We can observe that, since by definition $\vy_k$ and $\vy_{k-1}$ only differ over the coordinate block $S^{j_k},$ we have
\begin{equation}\label{eq:equiv-def-of-z_k}
    \begin{aligned}
        \vz_k &= \vz_{0} + \sum_{i=1}^k\mA^T(\vy_i - \vy_{i-1}) = \mA^T\vy_k. 
\end{aligned}
\end{equation}
On the other hand, using Eq.~\eqref{eq:equiv-def-of-z_k}, the definition of $\vq_k$ implies
\begin{equation}\label{eq:equiv-def-q_k}
    \begin{aligned}
        \vq_k &= A_{k+1}\vc + a_1 \mA^T \vy_0 + \sum_{i=1}^k \mA^T\big[a_{i+1} \vy_i + m a_i(\vy_i - \vy_{i-1})\big] 
    \end{aligned}
\end{equation}
%
Using the definition of the proximal operator (see Eq.~\eqref{eq:prox}) and the definition of $\vx_k$ in Step~4 of Algorithm~\ref{alg:clvr-impl}, we have
\begin{equation}\label{eq:alg-1-opt-cond-x_k}
\begin{aligned}
    \vx_k &= \argmin_{\vx \in \gX} \bigg\{\frac{A_k}{\gamma}r(\vx) + \frac{1}{2}\Big\|\vx - \vx_0 + \frac{1}{\gamma}\vq_{k-1}\Big\|^2\bigg\}\\
    &= \argmin_{\vx \in \gX} \bigg\{{A_k}r(\vx) + \frac{\gamma}{2}\Big\|\vx - \vx_0 + \frac{1}{\gamma}\vq_{k-1}\Big\|^2\bigg\}. 
\end{aligned}
\end{equation}

Now let us consider the optimization problem that defines $\vxh_k.$ Assume for now that the definitions of $\vy_k$ and $\vyh_k$ agree (we justify this claim below). 
Observe first that the minimization problem defining $\vxh_k$ is independent of $\vu,$ so by unrolling the recursion for $\phi_k,$ we have 
%
\begin{align*}
    \vxh_k &= \argmin_{\vx \in \gX} \Bigg\{ \frac{\gamma}{2}\|\vx - \vx_0\|^2 + A_k r(\vx) + a_1 \innp{\vx,\vc+\mA^T \bar{\vy}_0 } \\ & \hspace{2cm} + \sum_{i=2}^k a_i \innp{\vx, \vc + \mA^T\Big(\vy_{i-1} + \frac{m a_{i-1}}{a_i}(\vy_{i-1} - \vy_{i-2})\Big)} \Bigg\} \\
    &= \argmin_{\vx \in \gX} \Bigg\{\frac{\gamma}{2}\|\vx - \vx_0\|^2 + A_k r(\vx) \\ & \hspace{2cm} + \innp{\vx, A_k \vc + a_1 \mA^T \vy_0 + \mA^T\Big(\sum_{i=2}^k \big[a_i \vy_{i-1} + {m a_{i-1}}(\vy_{i-1} - \vy_{i-2})\big]\Big)}\Bigg\}\\
    &= \argmin_{\vx \in \gX} \Big\{\frac{\gamma}{2}\|\vx - \vx_0\|^2 + A_k r(\vx) + \innp{\vx, \vq_{k-1}}\Big\}\\
    &= \argmin_{\vx \in \gX} \bigg\{{A_k}r(\vx) + \frac{\gamma}{2}\Big\|\vx - \vx_0 + \frac{1}{\gamma}\vq_{k-1}\Big\|^2\bigg\}\\
    &= \vx_k. 
\end{align*}
It remains to argue that the definitions of $\vy_k$ and $\vyh_k$ agree. First, observe that since the definitions of $\psi_k$ and $\psi_{k-1}$ differ only over block $S^{j_k}$, we have $\vyh_k^{S^j} = \vyh_{k-1}^{S^j}$ for all $j \neq j_k$. 
For $j = j_k,$ we have by unrolling the recursive definition of $\psi_k$ that 
%
\begin{align*}
    \vyh_k^{S^{j_k}} &= \argmin_{\vy^{S^{j_k}} \in \sR^{n_{j_k}}}\Big\{\frac{1}{2\gamma}\|\vy^{S^{j_k}} - \vy_0^{S^{j_k}}\|^2 - \sum_{i=1}^k \mathds{1}_{\{S^{j_i} = S^{j_k}\}} m a_i \innp{\vy^{S^{j_i}}, \mA^{S^{j_i}}\vx_i - \vb^{S^{j_i}}} \Big\}\\
    &= \vy_0^{S^{j_k}} +\gamma\sum_{i=1}^k \mathds{1}_{\{S^{j_i} = S^{j_k}\}} m a_i (\mA^{S^{j_i}}\vx_i - \vb^{S^{j_i}})\\
    &= \vyh_{k-1}^{S^{j_k}} + \gamma  m a_k (\mA^{S^{j_k}}\vx_k - \vb^{S^{j_k}})\\
    &= \vy_k^{S^{j_k}},
\end{align*}
as claimed.
\end{proof}

\begin{algorithm*}[ht!]
\caption{Coordinate Linear Variance Reduction (Analysis Version)}\label{alg:clvr-v2}
\begin{algorithmic}[1]
\STATE \textbf{Input: }  $\vx_0=\vx_{-1}\in\gX, \vy_0 = \bar{\vy}_0 \in \sR^n$, $m, \{S^{1}, S^{2}, \ldots, S^{m}\},  K, \gamma>0, \hat{L} >0.$
\STATE $\phi_0(\cdot) = \frac{ \gamma}{2}\|\cdot - \vx_0\|^2, \psi_{0}(\cdot) = \frac{1}{2\gamma}\|\cdot - \vy_0\|^2$.
\STATE $a_1 = A_1 = \frac{1}{2\hat{L}m}.$
\FOR{$k = 1,2,3,\ldots, K$}
\STATE  $\vx_{k} = \argmin_{\vx\in\gX}\{\phi_k(\vx) =  \phi_{k-1}(\vx) + a_k (\langle {\vx} - \vu,  \mA^T\bar{\vy}_{k-1} + \vc  \rangle+ r(\vx) - r(\vu) )\}.$ 
\STATE Pick $j_k$ uniformly at random in $[m].$
\STATE $\vy_{k} = \argmin_{\vy\in\sR^{{n}}}\{  \psi_k(\vy) =  \psi_{k-1}(\vy) + m a_k(-
\langle \vy^{S^{j_k}} - \vv^{S^{j_k}}, \mA^{S^{j_k}}\vx_{k} -  \vb^{S^{j_k}}\rangle)\}$.
\STATE $a_{k+1} = \frac{\sqrt{1 + \sigma A_{k}/\gamma}}{2\hat{L} m}, A_{k+1} = A_{k} + a_{k+1}.$
\STATE $\bar{\vy}_{k} = \vy_{k} +  \frac{ma_{k}}{a_{k+1}} (\vy_{k} - \vy_{k-1}). $
\ENDFOR
\STATE \textbf{return } $\tilde{\vx}_K := \frac{1}{A_K}\sum_{k=1}^K a_k  \vx_k, \tilde{\vy}_K = \frac{1}{A_K}\sum_{k=1}^K (a_k \vy_k + (m-1)a_k(\vy_k - \vy_{k-1}))$. 
\end{algorithmic}
\end{algorithm*}

In the following three lemmas, we bound  $\phi_k(\vx_k)$ and $\psi_k(\vy_k)$ below and above, which is then subsequently used to bound the primal-dual gap in Theorem~\ref{thm:clvr}. 
\begin{lemma}\label{lem:block-psi-phi-upper}
For all steps of Algorithm \ref{alg:clvr-v2} with $k\ge 1$, we have, $\forall (\vu, \vv)\in\gX\times\sR^n,$ 
\begin{eqnarray}
\phi_k(\vx_k) &\le& \frac{\gamma}{2}\| \vu - \vx_0\|^2 - \frac{\gamma + \sigma A_k }{2}\|\vu - \vx_k\|^2, \notag\\
\psi_k(\vy_k) &\le& \frac{1}{2\gamma}\|\vv - \vy_0\|^2 - \frac{1}{2\gamma}\|\vv - \vy_k\|^2.     \notag
\end{eqnarray}
\end{lemma}
\begin{proof}
By the definitions of $\psi_k(\vy)$ and $\phi_k(\vx)$ in Algorithm \ref{alg:clvr-v2}, it follows that, $\forall k\ge 1,$
\begin{equation}\label{eq:vrpda-phi_k-def}
\begin{aligned}
\phi_k(\vx) =\;& \sum_{i=1}^k a_i ( \langle \vx - \vu, \mA^T\bar{\vy}_{i-1}  + \vc \rangle + r(\vx) -r(\vu)) +  \frac{\gamma}{2}\|\vx - \vx_0\|^2,\end{aligned}
\end{equation}
and
\begin{equation}\label{eq:vrpda-psi_k-def}
\begin{aligned}
\psi_k(\vy) =\; &\sum_{i=1}^k m a_i\Big( - \innp{\vy^{S^{j_i}} - \vv^{S^{j_i}}, \mA^{S^{j_i}}\vx_{i} - \vb^{S^{j_i}}}   \Big) + \frac{1}{2\gamma}\|\vy - \vy_0\|^2. 
\end{aligned}
\end{equation}

Observe that, as function of $\vx$, $\phi_k(\vx)$ is $(\gamma+\sigma A_k)$-strongly convex. As, by definition, $\vx_k = \argmin_{\vx\in\gX}\phi_k(\vx),$ it follows that
\begin{align}
\phi_k(\vu) \geq \phi_k(\vx_k) + \frac{\gamma+\sigma A_k  }{2}\| \vu-\vx_k\|^2.
\end{align}
Now, writing $\phi_k(\vu)$ explicitly and rearranging the last inequality, the stated bound on $\phi_k(\vx_k)$ follows. 

As a function of $\vy,$ $\psi_k(\vy)$ is $1/\gamma$-strongly convex. 
The proof for the bound on $\psi_k$ uses the same argument and is omitted.  
\end{proof}

In the following proof,  for $k\ge 1$, let $\gF_{k}$ denote the natural filtration, containing all the randomness in the algorithm up to and including iteration $k.$ Recall that $\mA = \begin{psmallmatrix} \mA^{S^1}\\ \vdots\\  \mA^{S^{m}} \end{psmallmatrix}$, and let $\mA^{\bar{S}^{j}} (j\in[m])$ denote the matrix $\mA$ with its $S^{j}$ block of rows replaced by a zero block.

For convenience, for $k=1,2,\ldots,$ we define 
\begin{align}
\vyh_k = \vy_{k-1} + \gamma m a_k (\mA\vx_k - \vb).  \label{eq:y-hat}
\end{align}
Then from the definition of $\vy_k$ in Algorithm~\ref{alg:clvr-impl}, we have $\E[\vy_k - \vy_{k-1}|\gF_{k-1}] = \frac{1 }{m}(\vyh_k - \vy_{k-1}).$

Motivated by \cite{alacaoglu2019convergence},  we have the following result.
\begin{lemma}\label{lem:expec-max}
Given the sequences $\{\vy_k\}$ in Algorithm \ref{alg:clvr-v2} and $\{\vyh_k\}$ in Eq.~\eqref{eq:y-hat}, we define the sequence $\{\check{\vv}_k\}$ by
\[
\check{\vv}_k = (\vy_k - \vy_{k-1}) -  \frac{1}{m}(\vyh_k - \vy_{k-1}).
\]
Then for any $\vv\in\sR^d$ that may be correlated with the randomness in $\{\check{\vv}_i\}_{i=1}^k$, we have 
\begin{align}
\E\Big[ - \sum_{i=1}^k \innp{ \vv,  \check{\vv}_i}\Big]  \le      \E\Big[ \frac{1}{2} \|\vy_0 - \vv\|^2 + \frac{1}{2} \sum_{i=1}^{k}\|\vy_i - \vy_{i-1}\|^2\Big],  \label{eq:check-v-y}
\end{align}
where the expectation is taken over all the randomness in the history.
\end{lemma}

\begin{proof}
First, we prove a bound for $\E\Big[ \sum_{i=1}^k \|\check{\vv}_i\|^2  \Big]$. By the fact that $\E[\vy_i-\vy_{i-1}|\gF_{i-1}] = \frac{1}{m}(\vyh_i - \vy_{i-1})$, for each $\check{\vv}_i,$ we have  
$$\E[\|\check{\vv}_i\|^2|\gF_{i-1}] = \E\Big[\|\vy_i - \vy_{i-1}\|^2 | \gF_{i-1} \Big] - \E\Big[\big\| \frac{1}{m}\big(\hat{\vy}_i-\vy_{i-1}\big)\big\|^2|\gF_{i-1}\Big] \le\E[\|\vy_i - \vy_{i-1}\|^2|\gF_{i-1}].$$
Taking expectation on all the randomness in the history, for $\E\Big[\sum_{i=1}^k \|\check{\vv}_i\|^2\Big],$ we have 
\begin{equation}
\E\Big[ \sum_{i=1}^k \|\check{\vv}_i\|^2  \Big]  =   \E\Big[  \sum_{i=1}^k \E[\|\check{\vv}_i\|^2|\gF_{i-1}]\Big] 
\le\E\Big[  \sum_{i=1}^k \E[\|\vy_i - \vy_{i-1}\|^2|\gF_{i-1}]\Big] 
=   \E\Big[\sum_{i=1}^k\|\vy_i - \vy_{i-1}\|^2\Big]. \label{eq:check-v}
\end{equation}
 Then to prove \eqref{eq:check-v-y}, we define the sequence $\{ \check{\vy}_k \}_{k=0}^\infty$ by $\check{\vy}_0 = \vy_0$ and 
$\check{\vy}_k =  \check{\vy}_{k-1} - \check{\vv}_k$ for $k=1,2,\dotsc$.
By expanding $\frac{1}{2}\|\check{\vy}_k - \vv\|^2$,
\begin{align*}
\frac{1}{2}\|\check{\vy}_k - \vv\|^2 & = \frac{1}{2}\|\check{\vy}_{k-1} - \vv\|^2 -\langle \check{\vv}_k, \check{\vy}_{k-1} - \vv\rangle + \frac{1}{2}\|\check{\vv}_k\|^2 \\
\implies\;\; \langle \check{\vv}_k, \check{\vy}_{k-1} - \vv\rangle & = \frac{1}{2}\|\check{\vy}_{k-1} - \vv\|^2           - \frac{1}{2}\|\check{\vy}_k - \vv\|^2 + \frac{1}{2}\|\check{\vv}_k\|^2.
\end{align*}
By summing this expression and telescoping, we obtain
\begin{align}
\sum_{i=1}^k \langle\check{\vy}_{i-1} -\vv,  \check{\vv}_i\rangle  =\;&  \frac{1}{2}\|\check{\vy}_{0} - \vv\|^2   -\frac{1}{2}\|\check{\vy}_k - \vv\|^2 +\sum_{i=1}^k  \frac{1}{2}\|\check{\vv}_i\|^2 \nonumber\\
\le\;& \frac{1}{2}\|\check{\vy}_{0} - \vv\|^2 +\sum_{i=1}^k  \frac{1}{2}\|\check{\vv}_i\|^2 \nonumber\\
=\;& \frac{1}{2}\|\vy_{0} - \vv\|^2 +\sum_{i=1}^k  \frac{1}{2}\|\check{\vv}_i\|^2,
\label{eq:check-y-v-check-v}
\end{align}
It follows that
\begin{align}
\E\Big[ - \sum_{i=1}^k \innp{ \vv,  \check{\vv}_i}\Big]    =\;&     \E\Big[   \sum_{i=1}^k \innp{ \check{\vy}_{i-1} - \vv,  \check{\vv}_i} - \sum_{i=1}^k \innp{\check{\vy}_{i-1},  \check{\vv}_i}\Big]    \nonumber   \\
  \le \;&\E\Big[ \frac{1}{2} \|\vy_0 - \vv\|^2 + \sum_{i=1}^k\frac{1}{2}\|\check{\vv}_i\|^2 - \sum_{i=1}^k \innp{\check{\vy}_{i-1},  \check{\vv}_i}  \Big] \nonumber    \\
\le \;& \E\Big[ \frac{1}{2} \|\vy_0 - \vv\|^2 + \frac{1}{2} \sum_{i=1}^{k}\|\vy_i - \vy_{i-1}\|^2 - \sum_{i=1}^k \innp{\check{\vy}_{i-1},  \check{\vv}_i}\Big].  \label{eq:check-v-sum} 
\end{align}
where the first inequality is by Eq.~\eqref{eq:check-y-v-check-v} and the second inequality is by Eq.~\eqref{eq:check-v}.
To obtain the result \eqref{eq:check-v-y}, we use the facts that $\E [ \check{\vv}_i \, | \, \gF_{i-1}]=0$ and 
that $\check{\vy}_{i-1}\in\gF_{i-1}$  to obtain $\E\Big[\sum_{i=1}^k \innp{\check{\vy}_{i-1},  \check{\vv}_i}\Big]=0$.
\end{proof}

We emphasize that Lemma \ref{lem:expec-max} holds for any $\vv$ that may be even correlated with the randomness in the algorithm.

\begin{lemma}\label{lem:block-psi-phi-lower}
For all steps of Algorithm \ref{alg:clvr-v2} with $k\ge 1$,  
we have for all $(\vu, \vv) \in \gX \times \sR^n$ that
\begin{align}
\phi_k(\vx_k) \ge\; & \phi_{k-1}(\vx_{k-1}) + \frac{\gamma+\sigma A_{k-1}}{2}\|\vx_k - \vx_{k-1}\|^2  + a_k(r(\vx_k) - r(\vu))  \notag  \\
\;& \quad -a_k \innp{\vx_k - \vu,  \mA^T(\vy_k - \vy_{k-1}) } + a_k \innp{\vx_k - \vu,  \mA^T\vy_k+ \vc }       \nonumber\\  
\;& \quad + m a_{k-1}\Big(\innp{\vx_{k-1} - \vu, \mA^T(\vy_{k-1} - \vy_{k-2})} +      \innp{\vx_k - \vx_{k-1}, \mA^T(\vy_{k-1} - \vy_{k-2})}\Big),\nonumber  \\
\psi_k(\vy_k) \ge\;&  \psi_{k-1}(\vy_{k-1}) +  \frac{1}{2\gamma} \|\vy_k - \vy_{k-1}\|^2   
  -  a_k \langle\mA \vx_k - \vb , \vy_k - \vv\rangle  \nonumber \\ \;& - (m-1)a_k  \langle \mA \vx_k - \vb, \vy_k - \vy_{k-1}\rangle
   + \frac{1}{\gamma} \innp{\vy_{k-1} - \vv, - (\vy_k - \vy_{k-1}) +  \frac{1}{m}(\vyh_k - \vy_{k-1})}.   \nonumber
\end{align}
\end{lemma}

\begin{proof}
For the first claim, we have from the definition of $\phi_k(\vx_k)$, using that $\phi_{k-1}(\vx_{k-1})$ is $(\gamma + \sigma A_{k-1})$-strongly convex and minimized at $\vx_{k-1},$ that
\begin{equation}\label{eq:vr-phi-0}
\begin{aligned}
\phi_k(\vx_k) \ge\; & \phi_{k-1}(\vx_{k-1}) + \frac{\gamma+\sigma A_{k-1}}{2}\|\vx_k - \vx_{k-1}\|^2  \\
&+ a_k\big(\langle \vx_k - \vu, \mA^T\bar{\vy}_{k-1} +\vc\rangle + r(\vx_k) - r(\vu)\big).  
\end{aligned}
\end{equation}
Meanwhile, by the definition of $\{\bar{\vy}_k\}$ (using $\vy_{-1}=\vy_0$ for the case of $k=1$),  we have that
\begin{align}
&a_k\innp{\vx_k - \vu,  \mA^T\bar{\vy}_{k-1}} \nonumber\\   
=\;& a_k\innp{\vx_k - \vu,      \mA^T\Big( \vy_{k-1} + \frac{ma_{k-1}}{a_k}(\vy_{k-1} - \vy_{k-2})  \Big) }             \nonumber\\   
=\;&- a_k \innp{\vx_k - \vu,  \mA^T(\vy_k - \vy_{k-1})} + a_k \innp{\vx_k - \vu,  \mA^T\vy_k}     \nonumber\\   
\;& + m a_{k-1}\Big(\innp{\vx_{k-1} - \vu, \mA^T(\vy_{k-1} - \vy_{k-2})} +      \innp{\vx_k - \vx_{k-1}, \mA^T(\vy_{k-1} - \vy_{k-2})}\Big).  \label{eq:vr-phi-1}
\end{align}
The claimed lower bound on $\phi_k(\vx_k)$ follows when we combine \eqref{eq:vr-phi-0} and \eqref{eq:vr-phi-1}.


%
For the second claim, we have  by the definition of $\psi_k$ that 
\begin{equation} 
\psi_k(\vy_k) - \psi_{k-1}(\vy_{k-1})
=  \psi_{k-1}(\vy_k) - \psi_{k-1}(\vy_{k-1}) \\
- m a_k   \innp{\vy^{S^{j_k}}_k - \vv^{S^{j_k}}, \mA^{S^{j_k}}\vx_{k}- \vb^{S^{j_k}}}.    \label{eq:block-psi-phi-lower-1}
\end{equation}
To obtain the claimed lower bound on $\psi_k$, we proceed to bound the terms on the right-hand side in \eqref{eq:block-psi-phi-lower-1}. 
First, since  $\psi_{k-1}$ is $(1/\gamma)$-strongly convex and minimized at $\vy_{k-1},$ we have
\begin{equation}
     \psi_{k-1}(\vy_k) - \psi_{k-1}(\vy_{k-1}) \geq \frac{1}{2\gamma}\|\vy_k - \vy_{k-1}\|^2. \label{eq:block-psi-phi-lower-3}
\end{equation}
Second, by using the definition of $\mA^{S_{j_k}}$ and ${\mA}^{\bar{S}_{j_k}},$ and by using several times that $\vy_{k-1}$ and $\vy_k$ differ only in their $S^{j_k}$ components, we have  
%

\begin{align}
- &\innp{\vy_k^{S^{j_k}} - \vv^{S^{j_k}}, \mA^{S^{j_k}}\vx_{k} - \vb^{S^{j_k}} }   \nonumber \\ 
&\hspace{1cm}= -\innp{\vy_k  - \vv  , \mA \vx_{k} - \vb} + \innp{\vy_k - \vv, \mA^{\bar{S}^{j_k}} \vx_{k} - \vb^{\bar{S}^{j_k}}}       \nonumber \\  
&\hspace{1cm}=   -\innp{\vy_k  - \vv  , \mA \vx_{k} - \vb} + \innp{\vy_{k-1} - \vv, \mA^{\bar{S}^{j_k}} \vx_{k} - \vb^{\bar{S}^{j_k}}}       \nonumber \\  
&\hspace{1cm}=   -\innp{\vy_k  - \vv  , \mA \vx_{k} - \vb} + \frac{m-1}{m} \innp{\vy_{k-1} - \vv, \mA \vx_{k} - \vb}       \nonumber \\  
&\hspace{1.5cm} + \innp{\vy_{k-1} - \vv, \mA^{\bar{S}^{j_k}} \vx_{k} - \vb^{\bar{S}^{j_k}} - \frac{m-1}{m}(\mA \vx_{k} - \vb)}      \nonumber  \\
&\hspace{1cm}= -\frac{1}{m}\innp{\vy_k  - \vv  , \mA \vx_{k} - \vb} + \frac{m-1}{m} \innp{\vy_{k-1} - \vy_k, \mA \vx_{k} - \vb}       \nonumber \\ 
&\hspace{1.5cm} + \innp{\vy_{k-1}^{S^{j_k}} - \vv^{S^{j_k}}, - (\mA^{S^{j_k}} \vx_{k} - \vb^{S^{j_k}})} + \frac{1}{m} \innp{ \vy_{k-1}-\vv, \mA \vx_{k} - \vb)}            \nonumber \\    
&\hspace{1cm}= -\frac{1}{m}\innp{\vy_k  - \vv  , \mA \vx_{k} - \vb} + \frac{m-1}{m} \innp{\vy_{k-1} - \vy_k, \mA \vx_{k} - \vb}       \nonumber \\ 
&\hspace{1.5cm} + \innp{\vy_{k-1} - \vv, - \frac{1}{\gamma m a_k}  (\vy_k - \vy_{k-1}) +  \frac{1}{\gamma m^2 a_k}(\vyh_k - \vy_{k-1})},
\label{eq:block-psi-phi-lower-4}
\end{align}
where in the last equality we used $\vy_k^{S^{j_k}} - \vy_{k-1}^{S^{j_k}} = \gamma m a_k (\mA^{S^{j_k}} \vx_{k} - \vb^{S^{j_k}})$ and $\vyh_k - \vy_{k-1} = \gamma m a_k (\mA\vx_k - \vb)$, which both hold by definitions of $\vyh_k$ and $\vy_k.$ 
 
Finally, combining \eqref{eq:block-psi-phi-lower-1}--\eqref{eq:block-psi-phi-lower-4}, we have the bound on $\psi_k(\vy_k)$ from the statement of the lemma. 
\end{proof}

By combining the two lower bounds in Lemma~\ref{lem:block-psi-phi-lower}, we obtain the following result.
\begin{lemma}\label{lem:vr-psi-phi}
For any $(\vu,\vv) \in \gX \times \sR^n$, the sum of $\psi_k(\vy_k) + \phi_k(\vx_k)$ can be bounded as follows: for all $k\ge 1,$
\begin{equation} 
\begin{aligned}
&\phi_k(\vx_k) + \psi_k(\vy_k)   \\
 \ge\;&     \phi_{k-1}(\vx_{k-1})  + \psi_{k-1}(\vy_{k-1})  \\
& \quad - m a_k\innp{\vx_k - \vu,  \mA^T(\vy_k - \vy_{k-1})} +  m a_{k-1} \innp{\vx_{k-1} - \vu, \mA^T(\vy_{k-1} - \vy_{k-2})}  \\ 
& \quad + \frac{1}{2\gamma} \|\vy_k - \vy_{k-1}\|^2 - \frac{1}{4\gamma}\|\vy_{k-1} - \vy_{k-2}\|^2 +Q_k,
\end{aligned} 
\end{equation}
where
\begin{equation}
\begin{aligned}
 Q_k :=& a_k\Big( r(\vx_k) - r(\vu)  -  \langle \mA \vu, \vy_k\rangle   +  \langle \vx_k, \mA^T \vv\rangle  +  \innp{\vx_k - \vu, \vc }   +  \langle  \vb , \vy_k - \vv\rangle  \\
 &\quad\quad - (m-1)\langle \mA  \vu - \vb, \vy_k - \vy_{k-1}\rangle\Big) + \frac{1}{\gamma} \innp{\vy_{k-1} - \vv, - (\vy_k - \vy_{k-1}) +  \frac{1}{m}(\vyh_k - \vy_{k-1})}.    
\end{aligned}\label{eq:Qk}    
\end{equation}
\end{lemma}
\begin{proof}
Before our proof, as $A_{-1}$ and $A_0$ are not used in Algorithm \ref{alg:clvr-v2}, without loss of generality, we set $A_{-1} = A_0 = 0.$ Fix any $(\vu, \vv) \in \gX \times \sR^n$. 
By combining the bounds on $\phi_k(\vx_k)$ and $\psi_k(\vy_k)$ from Lemma \ref{lem:block-psi-phi-lower}, we have $\forall k \ge 1$ that 
\begin{equation}\label{eq:vr-phi-psi-2}
\begin{aligned}
&\phi_k(\vx_k) + \psi_k(\vy_k)   \\
 \ge\;&    \phi_{k-1}(\vx_{k-1}) +  \psi_{k-1}(\vy_{k-1})   \\
& \quad  - m a_k  \innp{\vx_k - \vu,  \mA^T(\vy_k - \vy_{k-1})}   +  m a_{k-1} \innp{\vx_{k-1} - \vu, \mA^T(\vy_{k-1} - \vy_{k-2})}  \\ 
& \quad +P_k  + Q_k,
\end{aligned} 
\end{equation}
where 
\begin{equation} \label{eq:vr-P_k-1}
    \begin{aligned}
        P_k =  \frac{\gamma + \sigma A_{k-1}}{2}\|\vx_k - \vx_{k-1}\|^2+\frac{1}{2\gamma}\|\vy_k - \vy_{k-1}\|^2+  m a_{k-1} \innp{\vx_k - \vx_{k-1}, \mA^T(\vy_{k-1} - \vy_{k-2})}
    \end{aligned}
\end{equation}
and $Q_k$ is defined in Eq.~\eqref{eq:Qk}.

To find a lower bound on $P_k$ we start by bounding the magnitude of the inner product term in \eqref{eq:vr-P_k-1}. Recall that $\vy_{k-2}$ and $\vy_{k-1}$ differ only on the coordinate block $S^{j_{k-1}}.$ We thus have:
\begin{align}
 & | ma_{k-1}\langle \vx_{k} - \vx_{k-1}, \mA^T(\vy_{k-1} - \vy_{k-2})\rangle |\nonumber \\
=\; & | ma_{k-1}\ \langle\mA^{S^{j_{k-1}}}(\vx_{k} - \vx_{k-1}), \vy_{k-1}^{S^{j_{k-1}}} - \vy_{k-2}^{S^{j_{k-1}}}\rangle |\nonumber \\
\leq\; &ma_{k-1}\|\mA^{S^{j_{k-1}}}(\vx_{k} - \vx_{k-1})\|\|\vy_{k-1}^{S^{j_{k-1}}} - \vy_{k-2}^{S^{j_{k-1}}}\| \nonumber  \\
\leq\; &  (ma_{k-1})^2\gamma\|\mA^{S^{j_{k-1}}}(\vx_{k} - \vx_{k-1})\|^2 + \frac{1}{4\gamma}\|\vy_{k-1}^{S^{j_{k-1}}} - \vy_{k-2}^{S^{j_{k-1}}}\|^2 \nonumber   \\ 
\leq\; &  (m \hat{L} a_{k-1})^2\gamma   \|\vx_{k} - \vx_{k-1}\|^2 + \frac{1}{4\gamma}\|\vy_{k-1}^{S^{j_{k-1}}} - \vy_{k-2}^{S^{j_{k-1}}}\|^2 \nonumber      \\ 
=\; &  (m \hat{L}a_{k-1})^2\gamma   \|\vx_{k} - \vx_{k-1}\|^2 + \frac{1}{4\gamma}\|\vy_{k-1} - \vy_{k-2}\|^2
\label{eq:main-thm-1}
\end{align}
where the second inequality is  by Young's inequality, and the third inequality is by Assumption~\ref{ass:L-hat} where $\|\mA^{S^{j_{k-1}}}(\vx_{k} - \vx_{k-1})\| \le \hat{L} \|\vx_{k} - \vx_{k-1}\|.$
With our setting $a_{k-1} = \frac{\sqrt{1+\sigma A_{k-2}/\gamma }}{2m\hat{L}}\le \frac{\sqrt{1+\sigma A_{k-1}/\gamma }}{2m\hat{L}}$, for all $k \ge 1$, we have
\[
(m \hat{L} a_{k-1})^2 \gamma = \frac{\gamma + \sigma A_{k-1}}{4}.
\]
By substituting this equality into \eqref{eq:main-thm-1} and then combining with \eqref{eq:vr-P_k-1}, we obtain
\begin{align}
    P_k \geq \;&  \frac{\gamma+\sigma A_{k-1}}{4}  \|\vx_{k} - \vx_{k-1}\|^2 +  \frac{1}{2\gamma}\|\vy_k - \vy_{k-1}\|^2 - \frac{1}{4\gamma}\|\vy_{k-1} - \vy_{k-2}\|^2\nonumber  \\
     \geq \;&  \frac{1}{2\gamma}\|\vy_k - \vy_{k-1}\|^2 - \frac{1}{4\gamma}\|\vy_{k-1} - \vy_{k-2}\|^2. \label{eq:vr-P_k-final}
\end{align}
We complete the proof by combining Eqs.~\eqref{eq:vr-phi-psi-2}, \eqref{eq:vr-P_k-1} and the bound Eq.~\eqref{eq:vr-P_k-final} for $P_k$.
\end{proof}

By telescoping the inequality in Lemma~\ref{lem:vr-psi-phi}, and using Lemmas~\ref{lem:block-psi-phi-upper} and \ref{lem:expec-max}, we obtain the next result.  
\begin{lemma}\label{lem:gap-no-expectation}
For all $(\vu, \vv)\in\gX\times\sR^n,$ we have     
\begin{equation}
\begin{aligned}
\;&  A_k(\gL(\tilde{\vx}_k, \vv) - \gL(\vu, \tilde{\vy}_k))  + \frac{\gamma+ \sigma A_k}{4}\|\vu - \vx_k\|^2  + \frac{1}{2\gamma}\|\vv - \vy_k\|^2   \\
\leq \;& \frac{\gamma}{2}\|\vu - \vx_0\|^2 +   \frac{1}{2\gamma}\|\vv - \vy_0\|^2  - \frac{1}{4\gamma}\sum_{i=1}^{k}\|\vy_i - \vy_{i-1}\|^2   + \frac{1}{\gamma}\sum_{i=1}^k \innp{\vy_{i-1},  \check{\vv}_i}   -  \frac{1}{\gamma}\sum_{i=1}^k \innp{ \vv,  \check{\vv}_i}, 
\end{aligned} \label{eq:gap-1}
\end{equation}
where $\check{\vv}_i$ is defined in Lemma \ref{lem:expec-max}. 
\end{lemma}

\begin{proof}
Telescoping the inequality in Lemma \ref{lem:vr-psi-phi},  we have
\begin{align}
\phi_k(\vx_k) + \psi_k(\vy_k)  
\ge\;& \phi_{0}(\vx_{0}) +  \psi_{0}(\vy_{0}) \nonumber \\
& - m a_k \innp{\vx_k - \vu,  \mA^T(\vy_k - \vy_{k-1})} +  m a_{0}\innp{\vx_{0} - \vu, \mA^T(\vy_{0} - \vy_{-1})}  \notag\\
&  +  \frac{1}{2\gamma}\|\vy_k - \vy_{k-1}\|^2  -  \frac{1}{4\gamma}\|\vy_{0} - \vy_{-1}\|^2 + \frac{1}{4\gamma}\sum_{i=1}^{k-1}\|\vy_i - \vy_{i-1}\|^2    + \sum_{i=1}^k Q_i \nonumber  \\
=\;& - m a_k \innp{\vx_k - \vu,  \mA^T(\vy_k - \vy_{k-1})} + \frac{1}{4\gamma}\|\vy_k - \vy_{k-1}\|^2\notag\\
&+ \frac{1}{4\gamma}\sum_{i=1}^{k}\|\vy_i - \vy_{i-1}\|^2  + \sum_{i=1}^k Q_i,\label{eq:vr-psi-phi-4}
\end{align}
where the last equality is by the fact that $\psi_0(\vy_0) = \phi_0(\vx_0) = 0$, $a_0=0$, and our convention that $\vy_{-1}:=\vy_0$. 

%
Then by the convexity of $r(\cdot)$ and the definition of $\{Q_i\},$ we have 
\begin{align}
\sum_{i=1}^k Q_i \ge\;& A_k(r(\vxt_k) - r(\vu) - \langle \mA\vu, \vyt_k\rangle + \langle\vxt_k, \mA^T\vv\rangle + \langle \vxt_k - \vu, \vc\rangle + \langle \vb, \vyt_k - \vv\rangle)   \nonumber\\
\;&+ \frac{1}{\gamma}\sum_{i=1}^k \innp{\vy_{i-1} - \vv, - (\vy_i - \vy_{i-1}) +  \frac{1}{m}(\vyh_i - \vy_{i-1})}           \nonumber\\
=\;& A_k(\gL(\tilde{\vx}_k, \vv) - \gL(\vu, \tilde{\vy}_k)) +     \frac{1}{\gamma}\sum_{i=1}^k \innp{\vy_{i-1} - \vv, - \check{\vv}_i}.    \label{eq:vr-psi-phi-5}
\end{align}
where   $\tilde{\vx}_k = \frac{1}{A_k}\sum_{i=1}^k a_i \vx_i, \tilde{\vy}_k =  \frac{1}{A_k}\sum_{i=1}^k (a_i \vy_i + (m-1)a_i(\vy_i - \vy_{i-1}))$ (as defined in \eqref{eq:vxyk}) and the last equality is by the definition of the Lagrangian $\gL(\cdot, \cdot)$ and $\{\check{\vv}_i\}$ in Lemma \ref{lem:expec-max}.

Then by combining Eqs.~\eqref{eq:vr-psi-phi-4}-\eqref{eq:vr-psi-phi-5} and Lemma \ref{lem:block-psi-phi-upper}, we have 
\begin{align*}
\; &   A_k(\gL(\tilde{\vx}_k, \vv) - \gL(\vu, \tilde{\vy}_k))             \\
\leq \; &   \Big(\psi_k(\vy_k) + \phi_k(\vx_k)  +
m a_k \innp{\vx_k - \vu,  \mA^T(\vy_k - \vy_{k-1})}\Big) -  \frac{1}{4\gamma}\|\vy_k - \vy_{k-1}\|^2   \\ 
 \; & - \frac{1}{4\gamma}\sum_{i=1}^{k}\|\vy_i - \vy_{i-1}\|^2 + \frac{1}{\gamma}\sum_{i=1}^k \innp{\vy_{i-1},  \check{\vv}_i}    -  \frac{1}{\gamma}\sum_{i=1}^k \innp{ \vv,  \check{\vv}_i}  \\ 
 \leq \; &  \Big(    \frac{1}{2\gamma}\|\vv - \vy_0\|^2 - \frac{1}{2\gamma}\|\vv - \vy_k\|^2  +  \frac{\gamma}{2}\|\vu - \vx_0\|^2 - \frac{\gamma + \sigma A_k}{2}\|\vu - \vx_k\|^2 \Big)\\
&\quad\quad  +  m a_k \innp{\vx_k - \vu,  \mA^T(\vy_k - \vy_{k-1})}  -  \frac{1}{4\gamma}\|\vy_k - \vy_{k-1}\|^2  - \frac{1}{4\gamma}\sum_{i=1}^{k}\|\vy_i - \vy_{i-1}\|^2 \\
 \; & + \frac{1}{\gamma}\sum_{i=1}^k \innp{\vy_{i-1},  \check{\vv}_i}    -  \frac{1}{\gamma}\sum_{i=1}^k \innp{ \vv,  \check{\vv}_i}.    \\ 
\end{align*}


Finally, we have from the fact that $\vy_k$ and $\vy_{k-1}$ differ in only the $S^{j_k}$ components, Young's inequality, and the definition of $a_k$ that 
\begin{align}
 & | ma_{k}\langle \vx_{k} - \vu, \mA^T(\vy_k - \vy_{k-1})\rangle |\nonumber \\
=\; & | ma_{k}\ \langle\mA^{S^{j_{k}}}(\vx_{k} - \vu), \vy_{k}^{S^{j_{k}}} - \vy_{k-1}^{S^{j_{k}}}\rangle |\nonumber \\
\leq\; &ma_{k}\|\mA^{S^{j_{k}}}(\vx_{k} - \vu)\|\|\vy_{k}^{S^{j_{k}}} - \vy_{k-1}^{S^{j_{k}}}\| \nonumber  \\
\leq\; &  (ma_{k})^2\gamma\|\mA^{S^{j_{k}}}(\vx_{k} - \vu)\|^2 + \frac{1}{4\gamma}\|\vy_{k}^{S^{j_{k}}} - \vy_{k-1}^{S^{j_{k}}}\|^2 \nonumber   \\ 
\leq\; &    (m\hat{L}a_{k})^2\gamma \|\vx_{k} - \vu\|^2 + \frac{1}{4\gamma}\|\vy_{k}^{S^{j_{k}}} - \vy_{k-1}^{S^{j_{k}}}\|^2 \nonumber      \\ 
=\; &  (m\hat{L} a_{k})^2\gamma \|\vx_{k} - \vu\|^2 + \frac{1}{4\gamma}\|\vy_k - \vy_{k-1}\|^2\nonumber      \\ 
=\; &  \frac{\gamma + \sigma A_{k-1}}{4}\|\vx_{k} - \vu\|^2 + \frac{1}{4\gamma}\|\vy_k - \vy_{k-1}\|^2\nonumber\\
\leq\; &  \frac{\gamma + \sigma A_{k}}{4}\|\vx_{k} - \vu\|^2 + \frac{1}{4\gamma}\|\vy_k - \vy_{k-1}\|^2\nonumber
\end{align}
leading to
\begin{equation}
\begin{aligned}
\;&  A_k(\gL(\tilde{\vx}_k, \vv) - \gL(\vu, \tilde{\vy}_k)) + \frac{1}{2\gamma}\|\vv - \vy_k\|^2 + \frac{\gamma+ \sigma A_k}{4}\|\vu - \vx_k\|^2    \\
\leq \;& \frac{\gamma}{2}\|\vu - \vx_0\|^2 +   \frac{1}{2\gamma}\|\vv - \vy_0\|^2  - \frac{1}{4\gamma}\sum_{i=1}^{k}\|\vy_i - \vy_{i-1}\|^2  \notag  \\
    \;&     + \frac{1}{\gamma}\sum_{i=1}^k \innp{\vy_{i-1},  \check{\vv}_i}   -  \frac{1}{\gamma}\sum_{i=1}^k \innp{ \vv,  \check{\vv}_i}   
\end{aligned}
\end{equation}
and completing the proof. 
\end{proof}

\begin{lemma} \label{lem:sum-yk-x-y-star} 
Suppose that $(\vx^*,\vy^*)$ is a Nash point for \eqref{eq:pd-glp}.
Then the iterates $\vx_k, \vy_k$ from Algorithm~\ref{alg:clvr-impl} satisfy 
\begin{equation}\label{eq:sum-yk-x-y-star} 
\begin{aligned}
& \E\Big[\frac{\gamma+ \sigma A_k}{4}\|\vx^* - \vx_k\|^2 + \frac{1}{2\gamma}\|\vy^* - \vy_k\|^2 + \frac{1}{4\gamma}\sum_{i=1}^{k}\|\vy_i - \vy_{i-1}\|^2\Big]  \\
& \le   \frac{\gamma}{2}\|\vx^* - \vx_0\|^2 +   \frac{1}{2\gamma}\|\vy^* - \vy_0\|^2,   
\end{aligned}
\end{equation}
where the expectation is w.r.t.~all the randomness in the algorithm. 
\end{lemma}
\begin{proof}
Note that the existence of a Nash point is assumed in Assumption~\ref{assmpt:Nash}.
With $(\vu, \vv) = (\vx^*, \vy^*)$, by the definition of Nash equilibrium,  we have  
\begin{align}
 \gL(\tilde{\vx}_k, \vy^*) - \gL(\vx^*, \tilde{\vy}_k) =  (\gL(\tilde{\vx}_k, \vy^*) - \gL(\vx^*, \vy^*)) - (\gL(\vx^*, \tilde{\vy}_k) - \gL(\vx^*, \vy^*))  \ge 0. \label{eq:gap-2} 
\end{align}
By setting $(\vu,\vv)=(\vx^*,\vy^*)$ in the result of Lemma~\ref{lem:gap-no-expectation}, using \eqref{eq:gap-2} to eliminate the first term on the left-hand side of the inequality, and rearranging, we obtain
\begin{align*}
    & \frac{\gamma+ \sigma A_k}{4}\|\vx^* - \vx_k\|^2  + \frac{1}{2\gamma}\|\vy^* - \vy_k\|^2 +
    \frac{1}{4\gamma}\sum_{i=1}^{k}\|\vy_i - \vy_{i-1}\|^2  \\
    & \le \frac{\gamma}{2}\|\vx^* - \vx_0\|^2 +   \frac{1}{2\gamma}\|\vy^* - \vy_0\|^2  
    + \frac{1}{\gamma}\sum_{i=1}^k \innp{\vy_{i-1},  \check{\vv}_i}   -  \frac{1}{\gamma}\sum_{i=1}^k \innp{ \vv,  \check{\vv}_i}.
\end{align*}
The result will follow when we show that the expectation (with respect to all the randomness) of the last two terms on the right-hand side is zero. Since $\vy_{i-1}\in\gF_{i-1}$, we have 
\begin{align}
\E\left[\sum_{i=1}^k \innp{\vy_{i-1},  \check{\vv}_i}\right] =\;&  \E\left[\E\left[ \sum_{i=1}^k \innp{\vy_{i-1},  \check{\vv}_i}\Big| \gF_{i-1}\right]\right] =  \E\left[\sum_{i=1}^k \innp{\vy_{i-1},  \E\left[\check{\vv}_i| \gF_{i-1} \right] } \right]  = 0,   \label{eq:gap-3}
\end{align}
which takes care of the second-last term. For any fixed $\vv\in\sR^n,$ we also have 
\begin{align}
 \E\bigg[\frac{1}{\gamma}\sum_{i=1}^k \innp{ \vv,  \check{\vv}_i}\bigg] = 0, \label{eq:xstar-check-v}  
\end{align}
which takes care of the last term, and completes the proof.
\end{proof}

Next we state a technical lemma, proved in an earlier paper, which provides the final ingredient for the proof of Theorem~\ref{thm:clvr}. 

\begin{restatable}{lemma}{lemseqgrowth}[\cite{song2021variance}]\label{lemma:es-seq-growth}
Let $\{A_k\}_{k \geq 0}$ be a sequence of nonnegative real numbers such that $A_0 = 0$ and $A_k$ is defined recursively via $A_k = A_{k-1} +  \sqrt{c_1^2 + c_2A_{k-1}}$, where $c_1 > 0,$ and $c_2 \ge 0$.
Define $K_0 = \lceil\frac{c_2}{9c_1}\rceil.$ Then 
\begin{equation}\notag
A_k \geq
    \begin{cases}
        \frac{c_2}{9}\Big(k - K_0 + \max\Big\{{3\sqrt{\frac{c_1}{c_2}},\, 1\Big\}\Big)^2}, &\text{ if } c_2 > 0 \text{ and } k > K_0,\\
        c_1 k, &\text{ otherwise}.%
    \end{cases}
\end{equation}
\end{restatable}

We are now ready to prove our main result, which we restate here.
\mainthmvr*
\begin{proof}
To provide a guarantee in terms of primal-dual gap, we need to take the supremum on $\vu, \vv$ of $\{\gL(\tilde{\vx}_k, \vv) - \gL(\vu, \tilde{\vy}_k))\}$ over $\gW_k$; we denote the $\arg\sup$ by $(\hat{\vu},\hat{\vv})$.
When we subsequently take the expectation, we have to account for the fact that $\hat{\vv}$ will be correlated with the randomness in the iteration history. 
We can however, use Lemmas~\ref{lem:expec-max} and \ref{lem:sum-yk-x-y-star} to give the upper bound on  $\E\Big[-\sum_{i=1}^k \innp{ \hat{\vv},  \check{\vv}_i}\Big]$.



From \eqref{eq:gap-1}, using the fact that $\frac{1}{2\gamma}\|\vv - \vy_k\|^2 + \frac{\gamma+ \sigma A_k}{4}\|\vu - \vx_k\|^2 \ge 0$, we have 
\begin{align}
\;&\E\Big[  A_k \sup_{(\vu, \vv)\in\gW_k}\{\gL(\tilde{\vx}_k, \vv) - \gL(\vu, \tilde{\vy}_k)\})\Big] \nonumber   \\
\leq \;& \E\Big[  \frac{\gamma}{2}\|\hat{\vu} - \vx_0\|^2 +   \frac{1}{2\gamma}\| \vy_0-\hat{\vv}\|^2 \Big]  - \frac{1}{4\gamma}\E\Big[\sum_{i=1}^{k}\|\vy_i - \vy_{i-1}\|^2\Big]   + \frac{1}{\gamma}\E\Big[\sum_{i=1}^k \innp{\vy_{i-1},  \check{\vv}_i}\Big]   \nonumber   \\
    \;&     +  \frac{1}{\gamma}\E\Big[-\sum_{i=1}^k \innp{ \hat{\vv},  \check{\vv}_i}\Big] \nonumber   \\
\leq \;& \E\Big[  \frac{\gamma}{2}\|\hat{\vu} - \vx_0\|^2 +   \frac{1}{2\gamma}\| \vy_0-\hat{\vv}\|^2 \Big]  - \frac{1}{4\gamma}\E\Big[\sum_{i=1}^{k}\|\vy_i - \vy_{i-1}\|^2\Big]    \nonumber   \\
\;&  + \frac{1}{\gamma} \E\Big[\frac{1}{2}\|\vy_0-\hat{\vv}\|^2  + \frac{1}{2} \sum_{i=1}^{k}\|\vy_i - \vy_{i-1}\|^2\Big] \nonumber   \\
= \;&  \E\Big[  \frac{\gamma}{2}\|\hat{\vu} - \vx_0\|^2 +   \frac{1}{\gamma}\|\vy_0-\hat{\vv}\|^2 \Big]  + \frac{1}{4\gamma}\E\Big[\sum_{i=1}^{k}\|\vy_i - \vy_{i-1}\|^2\Big]\nonumber   \\
\leq \;&  \E\Big[  \frac{\gamma}{2}\|\hat{\vu} - \vx_0\|^2 +   \frac{1}{\gamma}\|\vy_0-\hat{\vv}\|^2 \Big] + \frac{\gamma}{2}\|\vx^* - \vx_0\|^2 +   \frac{1}{2\gamma}\|\vy^* - \vy_0\|^2. 
\label{eq:sup1}
\end{align}
where the first inequality is by  Eq.~\eqref{eq:gap-1}, the second inequality is by Lemma \ref{lem:expec-max} and \eqref{eq:gap-3}, and the last inequality is by Lemma~\ref{lem:sum-yk-x-y-star}.  This proves the first claim \eqref{eq:ks6}. 
The second claim  \eqref{eq:solution-distance} follows from Lemma \ref{lem:sum-yk-x-y-star} with the fact that $\frac{1}{4\gamma}\E\Big[\sum_{i=1}^{k}\|\vy_i - \vy_{i-1}\|^2\Big]\geq 0$.   
The final claim concerning $A_k$ follows from Lemma~\ref{lemma:es-seq-growth} when we set 
\[
c_1 = \frac{1}{2 \hat{L}m}, \quad c_2 = \frac{\sigma}{(2 \hat{L}m)^2 \gamma}.
\]
\end{proof}

%
\propobjconstraint*
\begin{proof}
Assume that  $\|\mA\tilde{\vx}_k - \vb\| \neq 0,$ as otherwise the first bound follows trivially. 
Let $\vu = \vx^*$ and $\vv = \frac{2\|\vy^*\|(\mA\tilde{\vx}_k - \vb)}{\|\mA \tilde{\vx}_k - \vb\|}$. Then we have 
\begin{align}
 & \gL(\tilde{\vx}_k, \vv) - \gL(\vu, \tilde{\vy}_k) \nonumber  \\    
= \;&(\vc^T \tilde{\vx}_k + r( \tilde{\vx}_k)   + 2\|\vy^*\| \|\mA\tilde{\vx}_k - \vb\|) - (\vc^T\vx^* + r(\vx^*)   +  \tilde{\vy}_k^T(\mA \vx^*   - \vb)) \nonumber \\
= \;& (\vc^T(\tilde{\vx}_k - \vx^*) + r( \tilde{\vx}_k) -  r(\vx^*)   )  +  2\|\vy^*\| \|\mA\tilde{\vx}_k  - \vb\|.  \label{eq:corollary1}
\end{align}

For any fixed $\vu$, and any $\vv \in \sR^n$ possibly depending on the randomness in the algorithm, we have from Lemma~\ref{lem:gap-no-expectation}, taking expectation over all the randomness, that 
\begin{align}
\;&  A_k\E[\gL(\tilde{\vx}_k, \vv) - \gL(\vu, \tilde{\vy}_k)] \nonumber  \\
\le \;& \E\left[ A_k(\gL(\tilde{\vx}_k, \vv) - \gL(\vu, \tilde{\vy}_k)) + \frac{1}{2\gamma}\|\vv - \vy_k\|^2 + \frac{\gamma+ \sigma A_k}{4}\|\vu - \vx_k\|^2\right] \nonumber    \\
\leq \;& \frac{\gamma}{2}\|\vu - \vx_0\|^2 +   \frac{1}{2\gamma} \E [ \|\vv - \vy_0\|^2 ] - \E\Big[\frac{1}{4\gamma}\sum_{i=1}^{k}\|\vy_i - \vy_{i-1}\|^2\Big] \nonumber    \\
    \;&     + \frac{1}{\gamma}\E\Big[\sum_{i=1}^k \innp{\vy_{i-1},  \check{\vv}_i}\Big]   -  \frac{1}{\gamma}\E\Big[\sum_{i=1}^k \innp{ \vv,  \check{\vv}_i}\Big].   \label{eq:corollary2}
\end{align}
Meanwhile, we have 
\begin{align*}
\nonumber
 &- \frac{1}{4\gamma}\E\Big[\sum_{i=1}^{k}\|\vy_i - \vy_{i-1}\|^2\Big]      -  \frac{1}{\gamma}\E\Big[\sum_{i=1}^k \innp{ \vv,  \check{\vv}_i}\Big]  \\
& \le \frac{1}{2\gamma}\E\Big[\|\vy_0 - \vv\|^2 \Big]   + \frac{1}{4\gamma} \E\Big[\sum_{i=1}^{k}\|\vy_i - \vy_{i-1}\|^2\Big]   \nonumber  \\
& \le \frac{1}{2\gamma}\E[\|\vy_0-\vv\|^2] + \frac{\gamma}{2}\|\vx^* - \vx_0\|^2 +   \frac{1}{2\gamma}\|\vy^* - \vy_0\|^2,  
\end{align*}
where the first inequality is by Lemma \ref{lem:expec-max} and the second inequality is by Lemma \ref{lem:sum-yk-x-y-star}.  
Since  $\vy_{i-1}\in\gF_{i-1}, $ we have as in \eqref{eq:gap-3} that 
\begin{align}
 \frac{1}{\gamma}\E\Big[\sum_{i=1}^k \innp{\vy_{i-1},  \check{\vv}_i}\Big] = 0.   \label{eq:corollary4}
\end{align}
By combining Eq.~\eqref{eq:corollary1}-\eqref{eq:corollary4} with $\vu = \vx^*$, we have
\begin{align}
\;&\E\big[(\vc^T(\tilde{\vx}_k - \vx^*) + r( \tilde{\vx}_k) -  r(\vx^*)   )  +  2\|\vy^*\| \|\mA\tilde{\vx}_k  - \vb\|\big]      \nonumber\\
\le\;&\frac{\gamma\|\vx^* - \vx_0\|^2 + \frac{1}{\gamma}\E[\|\vv - \vy_0\|^2] + \frac{1}{2\gamma}\|\vy^* - \vy_0\|^2}{A_k}.  \label{eq:corollary5}
\end{align}



By the KKT condition of \eqref{eq:pd-glp} and the optimality of $(\vx^*, \vy^*)$, we have for all $\vx\in\gX$ that
\begin{align}
(\vc^T\vx + r(\vx)) - (\vc^T\vx^* + r(\vx^*)) - \langle \vy^*, \mA \vx - \vb\rangle \ge 0,
\end{align}
and thus
\begin{align}
 (\vc^T\vx + r(\vx)) - (\vc^T\vx^* + r(\vx^*)) \ge -  \|\vy^*\|\| \mA \vx - \vb\|. \label{eq:obj-constraint-2}   
\end{align}
By combining Eq.~\eqref{eq:corollary5} and Eq.~\eqref{eq:obj-constraint-2}, we have 
\begin{align}
\E[\|\vy^*\|\cdot \|\mA\tilde{\vx}_k - \vb\|] \le \frac{\gamma\|\vx^* - \vx_0\|^2 + \frac{1}{\gamma}\E[\|\vv - \vy_0\|^2] + \frac{1}{2\gamma}\|\vy^* - \vy_0\|^2}{A_k},  \label{eq:corollary6}
\end{align}
proving our first claim. The second claim is obtained by combining Eqs.~\eqref{eq:corollary5}, \eqref{eq:obj-constraint-2}, and \eqref{eq:corollary6}.
\end{proof}

\subsection{Omitted proofs from Section \ref{sec:lazy}}\label{app:lazy}

We need to compute explicitly only those components of $\vq_{k-1}$ and $\vx_k$ that are needed to update $\vy_k$.
\begin{restatable}{lemma}{lemmaqk} \label{lem:qk}
For $\{\vq_k\}$ defined in Algorithm~\ref{alg:clvr-impl}, we have 
\begin{equation}
    \notag
    \vq_k = A_{k+1} (\vc+\vz_k) + \vr_k, \quad k=0,1,2,\dotsc,
\end{equation}
where $\vr_0=0$ and for all $k=1,2,\dotsc$
\begin{equation}
    \notag
    \vr_k = \vr_{k-1} + (ma_{k}-A_{k}) (\vz_{k}-\vz_{k-1}).
\end{equation}
\end{restatable}
\begin{proof}
The proof is by induction. For $k=0$, we have $\vq_0 = A_1(\vc+\vz_0)$ which is true by definition. Assuming that the result holds for some index $k$, we show that it continues to hold for $k+1$. We have from Step~\ref{line:a1l9} of Algorithm~\ref{alg:clvr-impl}, then using the inductive assumption, that
\begin{align*}
    \vq_{k+1} & = a_{k+2} (\vz_{k+1}+\vc) + ma_{k+1} (\vz_{k+1}-\vz_k) + \vq_k \\
    &= a_{k+2} (\vz_{k+1}+\vc) + ma_{k+1} (\vz_{k+1}-\vz_k) + A_{k+1} (\vc+\vz_k) + \vr_k \\
    &= A_{k+2} \vc + (A_{k+2}-A_{k+1}) \vz_{k+1} + ma_{k+1} (\vz_{k+1}-\vz_k) + A_{k+1} \vz_k + \vr_k \\
    &= A_{k+2} (\vc+\vz_{k+1}) + (ma_{k+1}-A_{k+1}) (\vz_{k+1}-\vz_k) +\vr_k \\
    &= A_{k+2} (\vc+\vz_{k+1}) + \vr_{k+1},
\end{align*}
as required.
\end{proof}

This lemma indicates that we can reconstruct $\vq_k$ (or any subvector of $\vq_k$ that we need, on demand), provided we maintain $\vz_k$ and $\vr_k$. Note that the update from $\vz_k$ to $\vz_{k+1}$ is sparse; these two vectors differ only in the components corresponding to the block $C^{j_k}$.
To obtain  $\vr_{k+1}$ from $\vr_k$, we need to add a scalar times the sparse vector $\vz_{k+1}-\vz_k$, so this update is also efficient.

We can also maintain an implicit representation of the averaged vector $\vyt_k$ efficiently, as shown in the following lemma. We defer the proofs of Lemmas~\ref{lem:qk} and~\ref{lem:yk} to Appendix~\ref{appx:proofs-general}.
\begin{restatable}{lemma}{lemmayk} \label{lem:yk}
For $\{\vyt_k\}$ defined in \eqref{eq:vxyk}, we have 
\begin{align}\notag
\vyt_k = \vy_k + \frac{1}{A_{k}}\vs_k, \quad k=1,2,\dotsc,     
\end{align}
where $\vs_0 = \vzero$ and for all $k=1,2,\dotsc,$
\begin{align}\notag
\vs_{k} = \vs_{k-1} + ((m-1)a_k- A_{k-1})(\vy_k - \vy_{k-1}).
\end{align}
\end{restatable}
%


\begin{proof}
Recall that  $\vyt_K\; (K\ge 1)$  is defined in  Step 11 of Algorithm \ref{alg:clvr-impl}. 
The proof is by induction. For $k=1$, $\vyt_1 = \vy_1 + \frac{1}{A_1}(m-1)a_1(\vy_1 - \vy_{0}) = \vy_1 + \frac{1}{A_1}\vs_1$ which is true by the definition of $\vs_1$ and $A_0$. Next, we assume that the result holds for some $k\ge 2$ and show that it continues to hold for $k+ 1$. We have
\begin{align}
A_{k+1} \vyt_{k+1} =\;& \sum_{i=1}^{k+1} (a_i \vy_i + (m-1)a_{i}(\vy_{i} - \vy_{i-1})) \nonumber \\  
=\;& A_k \vyt_{k} + a_{k+1} \vy_{k+1} + (m-1)a_{k+1}(\vy_{k+1} - \vy_{k})         \nonumber \\  
=\;&  A_{k} \vy_k + \vs_k       + a_{k+1} \vy_{k+1} + (m-1)a_{k+1}(\vy_{k+1} - \vy_{k})   \nonumber \\ 
=\;&  A_{k} (\vy_k - \vy_{k+1} + \vy_{k+1}) + \vs_k       + a_{k+1} \vy_{k+1} + (m-1)a_{k+1}(\vy_{k+1} - \vy_{k})   \nonumber \\ 
=\;&A_{k+1} \vy_{k+1} + \vs_k  + ((m-1)a_{k+1} - A_k) (\vy_{k+1} - \vy_{k}) \nonumber \\ 
=\;& A_{k+1} \vy_{k+1} + \vs_{k+1}.  \notag
\end{align}
as required.
\end{proof}

\subsection{Omitted proofs from Section~\ref{sec:restart}} \label{appendix:restart}
The standard-form LP \eqref{eq:pd-lp} is derived by setting $r(\vx) \equiv 0$ and $\gX = \{\vx \in \sR^d: x_i \geq 0, i \in [d]\}$  in \eqref{eq:pd-glp}.
Given any $\vw\in\gX\times\sR^n,$ we define a compact convex subset $\gW_{\zeta, \gamma}(\vw)$ of $\gX\times\sR^n$ as follows:
\begin{align}
 \gW_{\zeta, \gamma}(\vw) = \{\hat{\vw}\in\gX\times\sR^n: \|\hat{\vw} - \vw\|_{(\gamma)} \le \zeta, \zeta >0, \gamma > 0\}, \label{eq:w-r-gamma}
\end{align}
where we have defined $\|\cdot\|_{(\gamma)}$ by $\| \vw\|_{(\gamma)} = \sqrt{\gamma\|\vx\|^2 + \frac{1}{\gamma}\|\vy\|^2}$ in Section~\ref{sec:restart}. 

\begin{lemma}\label{lem:gap-metric}
Consider the standard-form LP  \eqref{eq:lp}.
Given $\tau>0$ and $\vw\in\gX\times\sR^n$, if $\|\vw\|_{(\gamma)}\le \tau$ and $\zeta\le \tau$, then
\begin{align}
 \sup_{\hat{\vw}\in\gW_{\zeta,\gamma}(\vw)}\{\gL(\vx, \vyh) - \gL(\vxh, \vy)\} \ge \frac{\zeta}{\gamma + 1/\gamma + \tau}\emph{LPMetric}(\vw).     
\end{align}
\end{lemma}
\begin{proof}
Let $\mF(\vw) = \begin{bmatrix}
 \mA^T\vy + \vc  \\ 
 \vb - \mA\vx 
\end{bmatrix}.$  
Then we have
\begin{align}
\rho_{\zeta,\gamma}(\vw) := \sup_{\vwh\in\gW_{\zeta,\gamma}(\vw)}  \{\gL(\vx, \vyh) - \gL(\vxh,\vy)\} =  \sup_{\vwh\in\gW_{\zeta,\gamma}(\vw)}  \mF(\vw)^T(\vw - \vwh)\ge 0, \label{eq:gap-metric0}
\end{align}
where the inequality follows from $\vw \in \gW_{\zeta,\gamma}(\vw)$.

First, we prove 
\begin{align}
\rho_{\zeta,\gamma}(\vw) \ge  \frac{\zeta\|\max\{-\mA^T\vy - \vc,\, \vzero\}\|_2}{\gamma}. \label{eq:gap-metric1}
\end{align}
If $\|\max\{-\mA^T\vy - \vc,\, \vzero\}\|_2 > 0,$ for $\vwh_1  = (\vxh_1, \vyh_1)$, let $\vxh_1 = \vx + \frac{\zeta}{\gamma \|\max\{-\mA^T\vy - \vc,\, \vzero\}\|_2} \max\{-\mA^T\vy - \vc,\, \vzero\}\in\gX$ and $\vyh_1 = \vy$. Then we have $\|\vwh_1 - \vw\|_{(\gamma)} = \zeta$ and thus $\vwh_1\in\gW_{\zeta,\gamma}(\vw)$. So, as $\rho_{\zeta,\gamma}(\vw) \ge \mF(\vw)^T(\vw - \vwh_1),$ with the definition of $\vwh_1,$ Eq.~\eqref{eq:gap-metric1} holds. If $\|\max\{-\mA^T\vy - \vc,\, \vzero\}\|_2 = 0,$ then Eq.~\eqref{eq:gap-metric1} holds trivially. 

Second, we prove 
\begin{align}
\rho_{\zeta,\gamma}(\vw) \ge  \zeta\gamma\|\mA\vx-\vb\|_2.  \label{eq:gap-metric2}
\end{align}
If $\|\mA\vx-\vb\|_2 > 0,$ for $\vwh_2  = (\vxh_2, \vyh_2)$, let $\vxh_2 = \vx$ and $\vyh_2 = \vy + \frac{\zeta \gamma}{\|\mA\vx-\vb\|}(\mA\vx-\vb)$. Then we have $\|\vwh_2 - \vw\|_{(\gamma)} = \zeta$ and thus $\vwh_2\in\gW_{\zeta,\gamma}(\vw)$. Then  as $\rho_{\zeta,\gamma}(\vw) \ge \mF(\vw)^T(\vw - \vwh_2),$  Eq.~\eqref{eq:gap-metric2} holds. If $\|\mA\vx-\vb\|_2  = 0,$ then Eq.~\eqref{eq:gap-metric2} holds trivially.

Third, we prove that
\begin{align}
\rho_{\zeta,\gamma}(\vw) \ge  \frac{\zeta}{\tau}\max\{\vc^T \vx + \vb^T\vy,\, \vzero\}.    \label{eq:gap-metric3}
\end{align}
Note that the inequality holds trivially if $\vc^T \vx + \vb^T\vy\le 0$, so we assume that  $\vc^T \vx + \vb^T\vy> 0$, and note that $\mF(\vw)^T\vw  = \vc^T \vx + \vb^T\vy> 0$ in this case. 
This condition implies that  $\vw \ne 0$, so we can define  $\vwh_3 = \vw - \min\big\{\frac{\zeta}{\|\vw\|_{(\gamma)}}, 1\big\}\vw$. 
Then we have $\|\vwh_3 - \vw\|_{(\gamma)} \le \frac{\zeta}{\|\vw\|_{(\gamma)}}\|\vw\|_{(\gamma)}\le \zeta.$ Meanwhile, we also have $\vxh_3\ge \vx_3 - \vx_3 =0,$ so that  $\vwh_3 \in \gX\times\sR^n.$ 
Thus, we have $\vwh_3\in\gW_{\zeta,\gamma}(\vw)$. 
Then, with the assumptions $\zeta\le \tau$ and $\|\vw\|_{(\gamma)}\le \tau,$ together with $\mF(\vw)^T\vw >0$, we have 
\begin{align}
\rho_{\zeta,\gamma}(\vw) \ge \min\Big\{\frac{\zeta}{\|\vw\|_{(\gamma)}}, 1\Big\}  \mF(\vw)^T\vw \ge \min\Big\{\frac{\zeta}{\tau}, 1\Big\}\mF(\vw)^T\vw = \frac{\zeta}{\tau}\mF(\vw)^T\vw = \frac{\zeta}{\tau}(\vc^T \vx + \vb^T\vy),  \label{eq:gap-metric4}
\end{align}
completing the proof of the claim \eqref{eq:gap-metric3}.



By combining Eqs.~ \eqref{eq:gap-metric1}, \eqref{eq:gap-metric2} and \eqref{eq:gap-metric3}, we obtain
\begin{align}
(\gamma + 1/\gamma + \tau)^2\rho_{\zeta,\gamma}^2(\vw)
\ge\;&\zeta^2(\|\max\{-\mA^T\vy - \vc,0\}\|_2^2 + \|\mA\vx-\vb\|_2^2 + \max\{\vc^T \vx + \vb^T\vy, 0\}^2) \nonumber\\  
\ge\;& \zeta^2 (\text{LPMetric}(\vw))^2,
\end{align}
from which the result follows.
%
%
%
\end{proof}


\begin{lemma}\label{lem:clvr-lem}
Consider the standard-form LP  \eqref{eq:lp}, and let $\vw^*=(\vx^*,\vy^*)$ be a Nash point (that is, a solution of the primal-dual form \eqref{eq:pd-lp}). 
For a starting point $\vw_0$, define $\zeta = \|\vw_0 - \vw^*\|_{(\gamma)}$  and choose $\gamma>0$.
Then for all $k=1,2,\dotsc$, we have 
\begin{align}
\E[\|\vwt_k - \vw^*\|_{(\gamma)}] \le \sqrt{2}\|\vw_0 - \vw^*\|_{(\gamma)},  \nonumber
\end{align}
where the expectation is taken w.r.t.~all the randomness up to iteration $k$. 
Further, for $\gW_{\zeta,\gamma}(\tilde\vw_k)$ defined as in \eqref{eq:w-r-gamma}, we have
\begin{align}
\E\Big[\sup_{\vw\in\gW_{\zeta,\gamma}(\tilde\vw_k)}\{\gL(\tilde{\vx}_k, \vy) - \gL(\vy, \tilde{\vy}_k)\}\Big]  
    \le \;&  \frac{25 \hat{L}m}{k} \|\vw_0-\vw^*\|_{(\gamma)}^2. 
\end{align}
\end{lemma}
\begin{proof}
By Theorem \ref{thm:clvr},  we have 
\[
\E\Big[\frac{\gamma}{4}\|\vx^* - \vx_k\|^2 + \frac{1}{2\gamma}\|\vy^* - \vy_k\|^2\Big] \le  \frac{\gamma}{2}\|\vx^* - \vx_0\|^2 +   \frac{1}{2\gamma}\|\vy^* - \vy_0\|^2 = \frac{1}{2}\|\vw_0 - \vw^*\|_{(\gamma)}^2,     
\]
by the definition of $\| \cdot \|_{(\gamma)}$. Using this definition again, we have that the left-hand side in this expression is bounded below by $\frac14 \| \vw_k - \vw^* \|^2_{(\gamma)}$, so that 
\begin{align}
\E\Big[\|\vw_k - \vw^*\|^2_{(\gamma)}] \le 2  \|\vw_0 - \vw^*\|_{(\gamma)}^2,  \quad  \forall k \geq 1. \label{eq:wk-w-start}
\end{align}
By convexity of $\| \cdot \|^2$, we obtain
\begin{align}
\E[\|\vwt_k - \vw^*\|_{(\gamma)}^2] = \E\Big[ \Big\|\frac{1}{k}\sum_{i=1}^k (\vw_i - \vw^*) \Big\|_{(\gamma)}^2\Big]\le \frac{1}{k}\E\Big[ \sum_{i=1}^k  \|\vw_i - \vw^*\|_{(\gamma)}^2\Big] \le 2 \|\vw_0 - \vw^*\|_{(\gamma)}^2.    \label{eq:w-hat-k-w-start}
\end{align}
The first claim of the lemma now follows by applying Jensen's inequality to this bound.


For 
$\vwh=(\vxh, \vyh) \in \gW_{\zeta,\gamma}(\tilde{\vw}_k)$,  we have the following bound on $\E\Big[\gamma\|\hat{\vx} - \vx_0\|^2 +   \frac{1}{\gamma}\|\hat{\vy} - \vy_0\|^2 \Big]$: 
\begin{align}
&\E\Big[\gamma\|\hat{\vx} - \vx_0\|^2 +   \frac{1}{\gamma}\|\hat{\vy} - \vy_0\|^2 \Big]  \nonumber   \\ 
=\;& \E\big[\|\vw_0-\hat{\vw}\|_{(\gamma)}^2\big]  \nonumber   \\ 
=\;& \E\big[\|\vw_0-\vw^* + \vw^* - \tilde{\vw}_k +\tilde{\vw}_k- \vwh\|_{(\gamma)}^2\big]  \nonumber   \\ 
\le \;&\E\big[(\|\vw_0-\vw^*\|_{(\gamma)} + \|\vw^* - \tilde{\vw}_k\|_{(\gamma)} + \|\tilde{\vw}_k- \vwh\|_{(\gamma)})^2\big]  \nonumber   \\ \le \;&\E\big[3(\|\vw_0-\vw^*\|_{(\gamma)}^2 + \|\vw^* - \tilde{\vw}_k\|_{(\gamma)}^2 + \|\tilde{\vw}_k- \vwh\|_{(\gamma)}^2)\big]  \nonumber   \\ 
\le \;& \E\big[3(\|\vw_0-\vw^*\|_{(\gamma)}^2 + 2\|\vw^* - \vw_0\|_{(\gamma)}^2 + \|\vw_0-\vw^*\|_{(\gamma)}^2 )\big]  \nonumber \\
\label{eq:kx1}
= \;&  12\|\vw_0-\vw^*\|_{(\gamma)}^2,
\end{align}
where the first inequality is by the triangle inequality, the second inequality is by the arithmetic inequality, the third inequality is by Eq.~\eqref{eq:w-hat-k-w-start} and our assumption $\hat{\vw} \in \gW_{\zeta, \gamma}(\tilde{\vw}_k)$ with $\zeta = \|\vw_0-\vw^*\|_{(\gamma)}$. 

For the case of linear programming, the strong convexity constant is $\sigma = 0$,  so that $A_k = \frac{k}{2 \hat{L} m}$ in  \clvr.
Thus, by applying   Theorem~\ref{thm:clvr} with $\gW_k = \gW_{\zeta,\gamma}(
\tilde{\vw}_k)$ and $\vwh = (\hat{\vx}, \hat{\vy}) = \arg\sup_{\vw = (\vx, \vy)\in\gW_{\zeta,\gamma}(
\tilde{\vw}_k)}\{\gL(\tilde{\vx}_k, \vy) - \gL(\vx, \tilde{\vy}_k)\}$, and using the definition of $\| \cdot \|_{(\gamma)}$ and Eq.~\eqref{eq:kx1}, we have 
\begin{align}
 \;&\E\Big[\sup_{\vw\in\gW_{\zeta,\gamma}(\tilde\vw_k)}\{\gL(\tilde{\vx}_k, \vy) - \gL(\vy, \tilde{\vy}_k)\}\Big] \nonumber   \\
 \le\;& \frac{2 \hat{L}m}{k}\bigg( \E\Big[\frac{\gamma}{2}\|\hat{\vx} - \vx_0\|^2 +   \frac{1}{\gamma}\|\hat{\vy} - \vy_0\|^2 \Big] +  \frac{\gamma}{2}\|\vx^* - \vx_0\|^2 +   \frac{1}{2\gamma}\|\vy^* - \vy_0\|^2\bigg) \nonumber   \\
  \le\;&   \frac{2 \hat{L}m}{k}\Big(\E\big[\|\vw_0 - \hat{\vw}\|^2_{(\gamma)}\big] + \frac{1}{2}\|\vw_0 - \vw^*\|_{(\gamma)}^2\Big)  \nonumber   \\
  \le \;& \frac{2 \hat{L}m}{k}\Big(12  \|\vw_0-\vw^*\|_{(\gamma)}^2  + \frac{1}{2}\|\vw_0 - \vw^*\|_{(\gamma)}^2\Big) \nonumber   \\   
  \le \;&  \frac{25 \hat{L}m}{k} \|\vw_0-\vw^*\|_{(\gamma)}^2.
\end{align}
\end{proof}


Then with Theorem \ref{thm:clvr}, Lemmas~\ref{lem:gap-metric} and \ref{lem:clvr-lem}, we give our theorem for the fixed restart strategy.  
\propLPMetric*

\begin{proof}


Applying Lemma \ref{lem:gap-metric} with $\vw = \tilde{\vw}_k,$ $\tau = \|\tilde{\vw}_k\|_{(\gamma)} + \|\vw_0 - \vw^*\|_{(\gamma)}$ and  $\zeta = \|\vw_0 - \vw^*\|_{(\gamma)} \le \tau$, we have 
\begin{align}
\sup_{\hat{\vw}\in\gW_{\zeta,\gamma}(\vw)}\{\gL(\vxt_k, \vyh) - \gL(\vxh, \vyt_k)\} \ge  \frac{\|\vw_0-\vw^*\|_{(\gamma)}\text{LPMetric}(\tilde{\vw}_k)}{\gamma + 1/\gamma +  \|\tilde{\vw}_k\|_{(\gamma)} + \|\vw_0 - \vw^*\|_{(\gamma)}}.     \label{eq:lem10app}
\end{align}

As $\frac{x^2}{y}$ is jointly convex in $x, y$ on the domain $\{(x,y): x\in\sR, y >0 \}$ \cite[Example 3.18]{boyd2004convex}
using Jensen's inequality, we have that $\E[\frac{x^2}{y}] \geq \frac{(\E[x])^2}{\E[y]}$. (In our case, this simply follows as $x, y$ depend on the same source of randomness.)
Applying this inequality to \eqref{eq:lem10app} with $x = \sqrt{\text{LPMetric}(\tilde{\vw}_k)}$ and $y = \gamma + 1/\gamma +  \|\tilde{\vw}_k\|_{(\gamma)} + \|\vw_0 - \vw^*\|_{(\gamma)}$, we obtain
\begin{align}
\E\Big[ \sup_{\hat{\vw}\in\gW_{\zeta,\gamma}(\tilde\vw_k)}\{\gL(\vxt_k, \vyh) - \gL(\vxh, \vyt_k)\}\Big] \ge\;&  \E\bigg[ \frac{\|\vw_0-\vw^*\|_{(\gamma)} (\sqrt{\text{LPMetric}(\tilde{\vw}_k)})^2}{\gamma + 1/\gamma +  \|\tilde{\vw}_k\|_{(\gamma)} + \|\vw_0 - \vw^*\|_{(\gamma)}}\bigg] \nonumber \\ 
\ge\;&  \frac{\|\vw_0-\vw^*\|_{(\gamma)} \Big(\E\big[\sqrt{\text{LPMetric}(\tilde{\vw}_k)}\big]\Big)^2}{\E\big[\gamma + 1/\gamma +  \|\tilde{\vw}_k\|_{(\gamma)} + \|\vw_0 - \vw^*\|_{(\gamma)}\big]}.  \label{eq:restart1}
\end{align}
Using Lemma \ref{lem:clvr-lem}, we have 
\begin{align}
\;&\E\big[\gamma + 1/\gamma +  \|\tilde{\vw}_k\|_{(\gamma)} + \|\vw_0 - \vw^*\|_{(\gamma)}\big]\nonumber \\ 
=\;&  \gamma + 1/\gamma +     \|\vw_0 - \vw^*\|_{(\gamma)} + \E[\|\tilde{\vw}_k\|_{(\gamma)}] \nonumber \\ 
\le\;&\gamma + 1/\gamma +     \|\vw_0 - \vw^*\|_{(\gamma)} + \E[\|\vw^*\|_{(\gamma)} + \|\tilde{\vw}_k - \vw^*\|_{(\gamma)}]  \nonumber \\
\le\;&\gamma + 1/\gamma +     \|\vw_0 - \vw^*\|_{(\gamma)} + \|\vw^*\|_{(\gamma)} + \sqrt{2}\|\vw_0 - \vw^*\|_{(\gamma)}   \nonumber \\
=\;&\gamma + 1/\gamma +    (\sqrt{2} + 1) \|\vw_0 - \vw^*\|_{(\gamma)} + \|\vw^*\|_{(\gamma)}. \label{eq:restart2}
\end{align}
By combining Eqs.~\eqref{eq:restart1} and \eqref{eq:restart2} and using the definition of $C_0$, we have 
\begin{align}
\E\Big[ \sup_{\hat{\vw}\in\gW_{\zeta,\gamma}(\tilde\vw_k)}\{\gL(\vxt_k, \vyh) - \gL(\vxh, \vyt_k)\}\Big] \ge\;& \frac{ \|\vw_0 - \vw^*\|_{(\gamma)}}{C_0} \Big(\E\big[\sqrt{\text{LPMetric}(\tilde{\vw}_k)}\big]\Big)^2   \nonumber \\
\ge\;&   \frac{ \|\vw_0 - \vw^*\|_{(\gamma)}}{C_0} \Big(\E\big[\sqrt{H_{\gamma} \text{dist}(\tilde{\vw}_k, \gW^*)_{(\gamma)}}\big]\Big)^2, \label{eq:restart3}
\end{align}
where the last inequality is by the definition of Hoffman constant in Eq. \eqref{eq:LP-Metric}. 
Meanwhile, by Lemma \ref{lem:clvr-lem}, we have: 
\begin{align}
\E\Big[ \sup_{\hat{\vw}\in\gW_{\zeta,\gamma}(\tilde\vw_k)}\{\gL(\vxt_k, \vyh) - \gL(\vxh, \vyt_k)\}\Big] \le\;&   \frac{25\hat{L}m}{k} \|\vw_0-\vw^*\|_{(\gamma)}^2. \label{eq:restart4}
\end{align}
Now, recalling that $\vw^* = \argmin_{\vw\in\gW^*}\|\vw_0 - \vw\|_{(\gamma)}$ and using Eq.~\eqref{eq:LP-Metric},  we have
\begin{equation}
    \label{eq:restart5}
 \|\vw_0 - \vw^*\|_{(\gamma)} = \text{dist}(\vw_0, \gW^*)_{(\gamma)}  \le \frac{1}{H_\gamma} \text{LPMetric}(\vw_0). 
\end{equation}
By combining Eqs.~\eqref{eq:restart3}-\eqref{eq:restart5}, we have 
\begin{align}
 \frac{1}{C_0} \Big(\E\Big[\sqrt{H_{\gamma} \text{dist}(\tilde{\vw}_k, \gW^*)_{(\gamma)}}\,\Big]\Big)^2  \le\;&  \frac{1}{C_0} \Big(\E\big[\sqrt{\text{LPMetric}(\tilde{\vw}_k)}\big]\Big)^2   \nonumber \\
  \le\;& \frac{25\hat{L}m}{k} \text{dist}(\vw_0, \gW^*)_{(\gamma)} \nonumber \\
  \le\;& \frac{25\hat{L}m}{H_{\gamma}k} \text{LPMetric}(\vw_0).  
\end{align}
Both bounds follow from this chain of inequalities.
 \end{proof}



%
%
Finally, we summarize the adaptive restart strategy in Algorithm~\ref{alg:restarted-CLVR}. 
\begin{algorithm}[ht!]
\caption{Lazy CLVR with Adaptive Restarts}\label{alg:restarted-CLVR}
\begin{algorithmic}[1]
\STATE \textbf{Input:} $\epsilon > 0,$ $\vx_0 \in \gX, \vy_0 \in \sR ^n,$ $ \gamma>0$, $\hat{L} >0$,  $K$, $\widehat{K}$, $m$, $\{S^{1}, S^{2}, \ldots, S^{m}\}$, $\{C^{1}, C^{2}, \ldots, C^{m}\}$
\STATE $t = 0,$ $\vx_0^{(0)} = \vx_0$, $\vy_0^{(0)} = \vy_0, $  $\vw^{(0)} = (\vx_0^{(0)}, \vy_0^{(0)})$
\REPEAT
\STATE Run Lazy CLVR (Algorithm~\ref{alg:clvr-lazy2}) until $\mathrm{LPMetric}(\vw^{(t+1)}) \leq \frac{1}{2}\mathrm{LPMetric}(\vw^{(t)})$ where, $\vw^{(t+1)} = \tilde{\vw}^{(t)}_K = (\tilde{\vx}_K^{(t)}, \tilde{\vy}_K^{(t)})$ and $\tilde{\vx}_K^{(t)}, \tilde{\vy}_K^{(t)}$ are the output points of Lazy CLVR
\STATE $t = t+1$
\UNTIL{$\mathrm{LPMetric}(\vw^{(t)}) \leq \epsilon$}
\STATE \textbf{Return:}  $\vw^{(t)}$
\end{algorithmic}
\end{algorithm}
%

\section{Omitted proofs from Section~\ref{sec:dro}}\label{appx:dro-proof}

\subsection{DRO with $f$-divergence-based ambiguity set}

\thmfdivergence*
\begin{proof}
Since the context is clear, to simplify the notation, in the following we use $\gP$ to denote $\gP_{\rho, n}.$ 
First, using Sion's minimax theorem, we have that
\begin{align*}
    \min_{\vx\in\gX}\sup_{\vp\in \gP} \sum_{i=1}^n p_i g(b_i\va_i^T\vx) = \sup_{\vp\in \gP}  \min_{\vx\in\gX}\sum_{i=1}^n p_i g(b_i\va_i^T\vx).
\end{align*}
Introducing auxiliary variables $\vw$ and $\vu$, the problem is further equivalent to
\begin{align*}
     \sup_{\vp\in \gP}   \min_{\substack{\vx\in\gX, \vw, \vu:\\
      u_i =  b_i\va_i^T\vx, i\in[n],\\
      g(u_i) \le w_i, i\in[n]}} \; \vp^T\vw \; \equiv\;  \min_{\substack{\vx\in\gX, \vw, \vu:\\
      u_i =  b_i\va_i^T\vx, i\in[n],\\
      g(u_i) \le w_i, i\in[n]}} \; \sup_{\vp\in \gP}\; \vp^T\vw,
\end{align*}
where the last equivalence is by applying minimax equality, which holds due to compactness of $\gP$~\cite{stoer1963duality}. 
Hence, we can conclude that
\begin{equation}\label{eq:dro-f-00}
    \min_{\vx\in\gX}\sup_{\vp\in \gP} \sum_{i=1}^n p_i g(b_i\va_i^T\vx) \; = \;  \min_{\substack{\vx\in\gX, \vw, \vu:\\
      u_i =  b_i\va_i^T\vx, i\in[n],\\
      g(u_i) \le w_i, i\in[n]}} \; \sup_{\vp\in \gP}\; \vp^T\vw. 
\end{equation}

For a fixed tuple $(\vx,\vw,\vu)$, using Lagrange multipliers to enforce the constraints from $\gP$, we can further write 
\begin{align}
&\; \sup_{\vp\in\gP} \vp^T\vw  =  - \inf_{\vp\in\gP} - \vp^T\vw  \notag  \\    
= \;& - \inf_{\vp\in\sR^n}\Big( - \vp^T\vw  + \sup_{\vv \geq 0,\gamma\in\sR, \mu \geq 0}\Big(-\vp^T\vv + \gamma\Big(\sum_{i=1}^n p_i - 1\Big)    + \mu\Big(D_f(\vp\|\mathbf{1}/n) - \frac{\rho}{n}\Big) \Big)\Big)  \notag  \\    
= \;&\sup_{\vp\in\sR^n} \inf_{\vv\ge 0,\gamma\in\sR, \mu\ge 0}\bigg( \vp^T\vw  +\vp^T\vv - \gamma\Big(\sum_{i=1}^n p_i - 1\Big)    - \mu\Big(D_f(\vp\|\mathbf{1}/n) - \frac{\rho}{n}\Big) \bigg). \notag 
\end{align}

Now, using the definitions of $D_f$ and the convex conjugate of $f$, we have 
\begin{align}
D_f(\vp\|\mathbf{1}/n) = \sum_{i=1}^n \frac{1}{n} f(n p_i) = \sum_{i=1}^n \frac{1}{n} \sup_{\nu_i\in\dom(f^*)}\big(n p_i\nu_i - f^*(\nu_i)\big).  \label{eq:dro-f-0}
\end{align}
As a result, we have 
\begin{align}
& \inf_{\mu\ge 0} \; - \mu \Big( D_f(\vp\|\mathbf{1}/n) - \frac{\rho}{n}\Big)   \notag \\
= &\; \inf_{\mu\ge 0}   -  \frac{\mu }{n} \Big( \sum_{i=1}^n  \sup_{\nu_i\in\dom(f^*), i\in[n]}\big(n p_i\nu_i - f^*(\nu_i)\big) - \rho\Big)           \notag \\
= &\; \inf_{\mu\ge 0, \nu_i\in\dom(f^*), i\in[n] }\; - \frac{\mu}{n}\Big(\sum_{i=1}^n \big(n p_i\nu_i - f^*(\nu_i)\big) - \rho \Big)         \notag \\
= &\;  \inf_{\mu\ge 0, q_i\in \mu \dom(f^*), i\in[n] }\;   \frac{1}{n}\sum_{i=1}^n \big(- p_i q_i +\mu  f^*\big(\frac{q_i}{n\mu}\big)\big) + \frac{\mu \rho}{n}  \notag \\ 
= &\;\inf_{\substack{\mu_1 = \mu_2=\cdots = \mu_n \ge 0, \\ q_i\in \mu_i \dom(f^*), i\in[n] }}\; \frac{1}{n} \sum_{i=1}^n  \big(- p_i q_i +\mu_i  f^*\big(\frac{q_i}{n \mu_i}\big)\big) +  \frac{\mu_1 \rho}{n}  , 
\notag
\end{align}
where the first equality is by Eq.~\eqref{eq:dro-f-0}, the third one is by the variable substitution $q_i  = n \mu \nu_i$, and the last one is by introducing $\mu_1, \mu_2, \ldots, \mu_n$ to replace $\mu$.  Then each $n\mu_i f^*\big(\frac{q_i}{n\mu_i}\big) (i\in[n])$  is a perspective function of $f^*$ \cite{boyd2004convex}, which is jointly convex w.r.t.~$(\mu_i, q_i).$  Hence, we can conclude that
\begin{align*}
    &\;\sup_{\vp\in\gP} \vp^T\vw \\
    =\; &\;\sup_{\vp\in\sR^n} \inf_{\substack{\gamma\in \sR, \vv\geq \vzero,\\\mu_1 = \mu_2=\cdots = \mu_n \ge 0, \\ q_i\in \mu_i \dom(f^*), i\in[n] }}\bigg( \vp^T\vw  +\vp^T\vv - \gamma\Big(\sum_{i=1}^n p_i - 1\Big)  +  \frac{1}{n} \sum_{i=1}^n  \big(- p_i q_i +\mu_i  f^*\big(\frac{q_i}{n\mu_i}\big)\big) + \frac{\mu_1 \rho}{n}\bigg)\\
    =\; &\; \inf_{\substack{\gamma\in \sR, \vv\geq \vzero,\\\mu_1 = \mu_2=\cdots = \mu_n \ge 0, \\ q_i\in \mu_i \dom(f^*), i\in[n] }} \sup_{\vp\in\sR^n} \bigg( \vp^T\vw  +\vp^T\vv - \gamma\Big(\sum_{i=1}^n p_i - 1\Big)  +  \frac{1}{n} \sum_{i=1}^n  \big(- p_i q_i +\mu_i  f^*\big(\frac{q_i}{n\mu_i}\big) \big) + \frac{ \mu_1\rho}{n}\bigg),
\end{align*}
where the last line is by strong duality. 
%
Thus, combining with  Eq.~\eqref{eq:dro-f-00}, we conclude that the original DRO problem with $f$-divergence based ambiguity set is equivalent to the following problem:
\begin{align*}
\min_{\vx, \vu,  \vv, \vw, \vmu, \vq, \gamma} & \max_{\vp\in\sR^n} \;\Big\{\gamma  + \frac{\rho\mu_1}{n} + \frac{1}{n}\sum_{i=1}^n  \mu_i f^*\Big(\frac{q_i}{\mu_i}\Big) + \vp^T\Big(\vw + \vv - \frac{\vq}{n} - \gamma\mathbf{1}_n\Big)  \Big\} \label{eq:dro-f-2}\\
& \begin{alignedat}{2}
\st\; &\;  u_i =  b_i\va_i^T\vx, &&i\in[n]                    \\
&\;   \mu_1 = \mu_2=\cdots = \mu_n ,  \quad  &&               \\
&\;      g(u_i) \le w_i, && i\in[n]                           \\
&\;    q_i\in \mu_i \dom(f^*),  && i\in[n]              \\
&\;     v_i\ge 0, \mu_i\ge 0, &&                   i\in[n]     \\
&\;      \vx\in\gX.
\end{alignedat}
\end{align*}
%
Finally, noticing that the maximization problem over $\vp \in \sR^n$ enforces the equality constraint $\vw + \vv - \frac{\vq}{n} - \gamma\mathbf{1}_n = \vzero_n,$ we obtain
\begin{align*}
    & \min_{\vx, \vu,  \vv, \vmu, \vq, \gamma} \, \Big\{\gamma  + \frac{\rho\mu_1}{n} + \frac{1}{n}\sum_{i=1}^n  \mu_i f^*\big(\frac{q_i}{\mu_i}\big)  \Big\} \\
   & \begin{alignedat}{2}
    \st\;& \;   \vw + \vv - \frac{\vq}{n} - \gamma\mathbf{1}_n = \vzero_n, && \\
\;& u_i =  b_i\va_i^T\vx, &&i\in[n]                    \\
&\;   \mu_1 = \mu_2=\cdots = \mu_n ,                   \\
&\;      g(u_i) \le w_i, && i\in[n]                           \\
&\;    q_i\in \mu_i \dom(f^*), && i\in[n]              \\
&\;     v_i\ge 0, \mu_i\ge 0, &&  i\in[n]     \\
&\;      \vx\in\gX, &&
    \end{alignedat}
\end{align*}
as claimed.
\end{proof}

\paragraph{Example: Conditional Value at Risk (CVaR) with hinge loss.} 

As a specific example of an application of Theorem~\ref{thm:f}, we consider CVaR at level $\alpha \in (0, 1),$ which leads to the optimization problem:
\begin{equation}
    \min_{\vx \in \sR^d} \sup_{\substack{\vp \geq 0, \vone_n^T\vp = 1\\
    p_i \leq \frac{1}{\alpha n} (i \in [d])}} \sum_{i=1}^n p_i g(b_i \va_i^T\vx),
\end{equation}
where $g(u_i) = \max \{0, 1 - u_i \}$ is the hinge loss. 
Here the ambiguity set constraint reduces to  simple bounds $ p_i \leq \frac{1}{\alpha n}$ for $i \in [n]$, so the reformulation based on the convex perspective function can be avoided altogether and replaced by simple Lagrange multipliers for this linear constraint. In particular, in the proof of Theorem~\ref{thm:f}, we can write
\begin{align*}
    \sup_{\vp \in \gP} \vp^T \vw &= \sup_{\vp\in\sR^n} \inf_{\vv\ge 0,\gamma\in\sR, \vmu\ge 0}\bigg( \vp^T\vw  +\vp^T\vv - \gamma\Big(\sum_{i=1}^n p_i - 1\Big)    - \sum_{i=1}^n\mu_i\Big(p_i - \frac{1}{\alpha n}\Big) \bigg)\\
    &= \sup_{\vp\in\sR^n} \inf_{\vv\ge 0,\gamma\in\sR, \vmu\ge 0}\bigg( \gamma + \vp^T(\vw  + \vv - \gamma \vone_n - \vmu )  + \frac{1}{\alpha n}\vmu^T \vone_n\Big) \bigg)\\
    &= \inf_{\vv\ge 0,\gamma\in\sR, \vmu\ge 0} \sup_{\vp\in\sR^n} \bigg( \gamma + \vp^T(\vw  + \vv - \gamma \vone_n - \vmu )  + \frac{1}{\alpha n}\vmu^T \vone_n\Big) \bigg). 
\end{align*}
Following the argument from the proof of Theorem~\ref{thm:f}, and expressing $w_i \ge g(u_i)$ equivalently as the pair of constraints $w_i \ge 0$ and $w_i \ge 1-u_i$, the problem reduces to
\begin{equation}\label{eq:cvar-hinge}
    \begin{aligned}
        \min_{\vx, \vu, \vv, \vw, \vmu, \gamma}\; &\; \Big\{\gamma + \frac{1}{\alpha n}\vmu^T \vone_n \Big\}\\
        \st \; &\; \vw  + \vv - \gamma \vone_n - \vmu = \vzero_n, &&\\
        &\; u_i = b_i \va_i^T \vx, && i \in [n],\\
        &\; w_i \geq 0, \; w_i \geq 1 - u_i, && i \in [n],\\
        &\; v_i \geq 0, \; \mu_i \geq 0,  && i \in [n], 
    \end{aligned}
\end{equation}
which is a linear program. 
To write it in the standard form, we further introduce slack variables $\vs \in \sR^n$, $\vs \geq \vzero,$ to replace inequality constraints $w_i \geq 1 - u_i$ by $s_i - u_i - w_i = -1 $. For implementation purposes, we define $\gX$ to be the set of simple non-negativity constraints ($w_i \geq 0, s_i \geq 0,$, $v_i \geq 0, \mu_i, \geq 0,$ $\forall i \in [n]$) from Eq.~\eqref{eq:cvar-hinge}. The problem then becomes
\begin{equation}\notag
    \begin{aligned}
        \min_{\vx, \vu, \vv, \vw, \vmu, \vs, \gamma}\; &\; \Big\{\gamma + \frac{1}{\alpha n}\vmu^T \vone_n \Big\}\\
        \st \; &\; \vw  + \vv - \gamma \vone_n - \vmu = \vzero_n, &&\\
        &\; u_i - b_i \va_i^T \vx = 0, && i \in [n],\\
        &\;  \vs - \vu - \vw = -\vone_n,\\
& \vw \ge \vzero_n, \; \vv \ge \vzero_n, \; \vmu \ge \vzero_n, \; \vs \ge \vzero_n.
    \end{aligned}
\end{equation}

\subsection{DRO with Wasserstein metric-based ambiguity set}\label{app:dro-wass}
\begin{definition}\label{def:wass}
Let $\vmu$ and $\vnu$ be two probability distributions supported on $\Theta = \sR^d\times\{1, -1\}$ and let $\Pi(\vmu, \vnu)$ denote the set of all joint distributions between $\vmu$ and $\vnu$. Then the Wasserstein metric between $\vmu$  and $\vnu$ is defined by
\begin{align} 
W(\vmu, \vnu) = \inf_{\pi\in\Pi(\vmu, \vnu)}\int_{\Theta\times\Theta}\zeta(\vxi, \vxi') \pi(d \vxi, d\vxi'), 
\end{align}
where $\vxi, \vxi'\in\Theta$ and  $\zeta(\cdot, \cdot):\Theta\times \Theta\rightarrow\sR_+$ is a convex cost function defined by 
\begin{align}\notag
\zeta((\va, b), (\va', b')) = \|\va - \va'\| + \kappa |b - b'|,    
\end{align}
where $\|\cdot\|$ denotes a general norm and $\kappa>0$ is used to balance the feature mismatch and label uncertainty.
Let $\sQ = \frac{1}{n}\sum_{i=1}^n \delta_{(\va_i, b_i)}$, where $\delta_{(\va_i, b_i)}$ is the Dirac Delta function (or a point mass) at point $(\va_i, b_i).$ Then, the Wasserstein metric based ambiguity set is defined as
\begin{align}
\gP_{\rho, \kappa} = \Big\{\sP: W(\sP, \sQ) \le \rho\Big\}.    \label{eq:wass-ambi}
\end{align}
\end{definition}

\begin{assumption}\label{assumption-dro-wasserstein-boundedness}
$\sR \ni M = \sup_{\theta \in \dom(g^*)}|\theta|.$ 
\end{assumption}

We first show the following auxiliary lemma which is used in the proof of Theorem~\ref{thm:dro-wass}. A similar result for the special case of a logistic loss function can be found in \cite{shafieezadeh2015distributionally}. 
\begin{lemma}\label{lem:aux}
Let $(\va', b')\in\sR^d\times\{1,-1\}$ be a given pair of data sample and label. Then, for every $\lambda>0$, we have 
\begin{equation}
\sup_{\va\in\sR^d}g(b'\va^T\vw) - \lambda\|\va - \va'\| =   
\begin{cases}
g(b'\va'^T\vw), & \text{ if }\;  \|\vw\|_*\le \lambda/M \\ 
+\infty, &              \text{ otherwise }
\end{cases}.  
\end{equation}
\end{lemma}
\begin{proof}
Since $g$ is assumed to be proper, convex, and lower semicontinuous, by the Fenchel-Moreau theorem, it is equal to its biconjugate. Applying this property, we have 
\begin{align*}
&\;\;\;\sup_{\va\in\sR^d} \quad\; \{g(b'\va^T\vw) - \lambda\|\va - \va'\|    \}       \\ 
 =\;&\sup_{\substack{\va\in\sR^d, \\ \theta\in\dom(g^*)}}\big\{ \theta b'\va^T\vw - g^*(\theta)  - \lambda\|\va- \va'\|    \big\} \\
    =\; & \sup_{\substack{\va\in\sR^d, \\ \theta\in\dom(g^*)}} \big\{ \theta b'(\va - \va')^T\vw + \theta b'(\va')^T\vw- g^*(\theta)  - \lambda\|\va- \va'\|.    \big\} 
\end{align*}
Applying the change of variable $\vv:= \va - \va',$ we further have
\begin{align*}
&\quad \sup_{\va\in\sR^d} \quad\{g(b'\va^T\vw) - \lambda\|\va - \va'\|  \}         \\ 
=\; & \sup_{\substack{\vv\in\sR^d, \\ \theta\in\dom(g^*)}} \big\{ \theta b'\vv^T\vw + \theta b'(\va')^T\vw- g^*(\theta)  - \lambda\|\vv\|    \big\}\\
=\; & \sup_{\theta\in\dom(g^*)} \big\{\mathds{1}_{\{\|\theta b'\vw\|_*\leq \lambda\}} + \theta b'(\va')^T\vw- g^*(\theta) \big\}\\
 =\; &
\begin{cases}
 g(b'(\va')^T\vw), & \text{ if } \sup_{\theta \in \dom(g^*)}\|\theta b'\vw\|_* \le \lambda \\ 
 +\infty,          & \text{ otherwise } 
 \end{cases}, 
\end{align*}
where the second equality is by the convex conjugate of a norm $\|\cdot\|$ being equal to the convex indicator of the unit ball w.r.t.~the dual norm $\|\cdot\|_*.$ Finally, it remains to use that, by Assumption~\ref{assumption-dro-wasserstein-boundedness}, $\sup_{\theta \in \dom(g^*)}|\theta| = M.$ 
\end{proof}

Then following~\cite[Theorem 1]{shafieezadeh2015distributionally},  we provide reformulations of the problem from Eq.~\eqref{eq:wass} that can be addressed by computationally efficient solvers. 
\begin{restatable}{theorem}{propwasserstein}\label{thm:dro-wass}
The optimization problem from Eq.~\eqref{eq:wass} is equivalent to:
\begin{equation}\label{eq:dro-wass}
\begin{alignedat}{2}
\min_{\vw, \lambda, \vu, \vv, \vs, \vt} &\;\; \rho\lambda + \frac{1}{n}\sum_{i=1}^n s_i && \\
\st \;\;&  u_i = b_i \va_i^T \vw,  \quad\quad && i\in[n],  \\
    \;\;&  v_i = -u_i,    && i\in[n],  \\
    \;\;& t_i = 2\kappa \lambda + s_i, &&  i\in[n],  \\ 
    \;\;& g(u_i) \le s_i, && i\in[n], \\
    \;\;& g(v_i) \le t_i, && i\in[n], \\
        & \|\vw\|_* \le \lambda/M. 
\end{alignedat}
\end{equation}
\end{restatable}
%

\begin{proof} 
Let $\vz=(\va, b) \in \Theta := \sR^d\times\{1,-1\}$ and let $h_{\vw}(\vz) := g(b\va^T\vw).$ Then by the definition of the Wasserstein metric, 
\begin{align}
\sup_{\sP\in\gP_{\rho,\kappa}}\E^{\sP}[g(b\va^T\vw)] = 
\begin{cases}
\sup_{\pi\in\Pi(\sP, \hat{\sP}_n)}&\int_{\Theta}h_{\vw}(\vz)\pi(\dd\vz, \Theta) \\ 
\st & \int_{\Theta\times\Theta}\zeta(\vz, \vz')\pi(\dd\vz, \dd\vz')\le \rho. 
\end{cases}\label{eq:wass-1}
\end{align}
Assume that the conditional distribution of $\vz$ given $\vz' = (\va_i, b_i)$ is $\sP^i$, for all $i\in[n]$. Then, based on the definition of $\sP_n = \frac{1}{n}\sum_{i=1}^n \delta_{(\va_i, b_i)}$, we have 
\begin{align}
\pi(\dd\vz, \dd\vz') = \frac{1}{n}\sum_{i=1}^n \delta_{(\va_i, b_i)} \sP^i(\dd\vz). 
\end{align}
As a result, the problem from Eq.~\eqref{eq:wass-1} is equivalent to 
\begin{align}
\sup_{\sP\in\gP_{\rho,\kappa}}\E^{\sP}[g(b\va^T\vw)] =  
\begin{cases}
\sup_{\sP^i}&\frac{1}{n}\sum_{i=1}^n \int_{\Theta}h_{\vw}(\vz)\sP^i(\dd\vz) \\ 
\st & \frac{1}{n}\sum_{i=1}^n \int_{\Theta}\zeta(\vz, \vz')\sP^i(\dd\vz)\le \rho \\ 
    &   \int_{\Theta}\sP^i(\dd \vz) = 1. 
\end{cases}\label{eq:wass-2}
\end{align}
Then substituting in $\vz = (\va, b)$, using that the domain of $y$ is $\{1, -1\}$, and decomposing $\sP^i$ into unnormalized measures $\sP_{\pm 1}(\dd \va) = \sP^i(\dd \va, \{b = \pm 1\})$ supported on $\sR^d$, the RHS of Eq.~\eqref{eq:wass-2} can be simplified to  
\begin{align}
\sup_{\sP_{\pm}^i}\;\; & \frac{1}{n}\sum_{i=1}^n\int_{\sR^d} \big(h_{\vw}(\va, 1)\sP_{1}^i(\dd\va) +  h_{\vw}(\va, -1)\sP_{-1}^i(\dd\va)\big) \nonumber\\
\st\;\; & \frac{1}{n}\sum_{i=1}^n\int_{\sR^d}\big(\zeta((\va, 1), (\va_i, \vb_i)) \sP_{1}^i(\dd\va) + \zeta((\va, -1), (\va_i, \vb_i)) \sP_{-1}^i(\dd\va)\big)\le\rho\nonumber\\ 
    & \int_{\sR^d} \big(\sP_{1}^i(\dd\va) +   \sP_{-1}^i(\dd\va)\big)  = 1. \notag 
\end{align}
With the definition of the cost function $\zeta((\va, b), (\va', b')) = \|\va - \va'\| + \kappa|b - b'|$, it follows that 
\begin{align}
&\frac{1}{n}\sum_{i=1}^n\int_{\sR^d}\zeta((\va, 1), (\va_i, \vb_i)) \sP_{1}^i(\dd\va) + \zeta((\va, -1), (\va_i, \vb_i)) \sP_{-1}^i(\dd\va)     \nonumber\\ 
=\;& \frac{1}{n}\int_{\sR^d}\sum_{b_i = 1}[ \|\va - \va_i\|\sP^i_{1}(\dd\va) + \|\va - \va_i\|\sP^i_{-1}(\dd\va) + 2\kappa\sP_{-1}^i(\dd\va)]  \nonumber\\ 
& +\frac{1}{n}\int_{\sR^d}\sum_{b_i = -1}[ \|\va - \va_i\|\sP^i_{-1}(\dd\va) + \|\va - \va_i\|\sP^i_{1}(\dd\va) + 2\kappa\sP_{1}^i(\dd\va)]  \nonumber\\ 
=\;& \frac{2\kappa}{n}\int_{\sR^d}\Big(\sum_{b_i = 1}\sP_{-1}^i(\dd\va) + \sum_{b_i = -1}\sP_{1}^i(\dd\va)\Big) +\frac{1}{n}\int_{\sR^d}\sum_{i=1}^n\|\va- \va_i\|(\sP^i_{-1}(\dd\va) +\sP^i_{1}(\dd\va)).     
\end{align}
Thus, we have 
\[
\sup_{\sP\in\gP_{\rho,\kappa}}\E^{\sP}[g(b\va^T\vw)] 
= \left\{
\begin{aligned}
 \sup_{\sP^i_{\pm 1}}\, &\frac{1}{n}\sum_{i=1}^n\int_{\sR^d} \big(h_{\vw}(\va, 1)\sP_{1}^i(\dd\va) +  h_{\vw}(\va, -1)\sP_{-1}^i(\dd\va)\big) \\ 
 \st \quad & \frac{2\kappa}{n}\int_{\sR^d}\left(\sum_{b_i = 1}\sP_{-1}^i(\dd\va) + \sum_{b_i = -1}\sP_{1}^i(\dd\va)\right)\\
 &\hspace{.5in}+\frac{1}{n}\int_{\sR^d}\sum_{i=1}^n\|\va- \va_i\|(\sP^i_{-1}(\dd\va) +\sP^i_{1}(\dd\va)) \le \rho  \nonumber\\ 
 & \int_{\sR^d} \big(\sP_{1}^i(\dd\va) +   \sP_{-1}^i(\dd\va)\big)  = 1.  
\end{aligned}
\right.
\]
The above problem w.r.t. $\sP^i_{\pm 1}$ is an infinite-dimensional linear program with a finite number of constraints. By \cite[Proposition~3.4]{shapiro2001duality}, we get the following equivalent dual formulation: 
\[
\sup_{\sP\in\gP_{\rho,\kappa}}\E^{\sP}[g(b\va^T\vw)] =  
\left\{
\begin{alignedat}{2}
\min_{\lambda, s_i} \,&  \rho    \lambda+ \frac{1}{n}\sum_{i=1}^n s_i &&\\ 
\st\, & \sup_{\va\in\sR^d} h_{\vw}(\va, 1) - \lambda \|\va - \va_i\| - \lambda \kappa(1-b_i) \le s_i,\;\; && i\in[n] \\
&\sup_{\va\in\sR^d} h_{\vw}(\va, -1) - \lambda \|\va - \va_i\| - \lambda \kappa(1+b_i) \le s_i, \;\; && i\in[n] \\
&\lambda \ge 0. &&
\end{alignedat}
\right.
\]
Then, recalling that $h_{\vw}(\va, \pm 1)$ is short for $g(\pm\va^T\vw)$, by Lemma \ref{lem:aux}, we have
\begin{align*}
 \sup_{\va\in\sR^d} h_{\vw}(\va, \pm 1) - \lambda \|\va - \va_i\| = 
 \begin{cases}
 g(\pm \va_i^T\vw),  &  \text{ if } \sup_{\theta\in\dom(g^*)}\|\theta\vw\|_* \le \lambda \\ 
 +\infty,           &       \text{ otherwise. }
 \end{cases}
\end{align*}

Finally, the resulting reformulation is 
\begin{alignat*}{2}
\min_{\vw, \lambda, \vs} \; \rho\lambda + \frac{1}{n}\sum_{i=1}^n s_i &&& \\ 
\st \quad\quad\quad\quad   g(b_i \va_i^T \vw) & \le s_i, && i\in[n], \\
\; g(-b_i \va_i^T \vw) - 2\lambda \kappa & \le s_i, \quad && i\in[n], \\
 \sup_{\theta\in\dom(g^*)} \|\theta\vw\|_* & \le \lambda.  &&
\end{alignat*}
Finally, recalling that, by assumption, $\sup_{\theta \in \dom(g^*)}|\theta| = M,$ it follows that the constraint $\sup_{\theta\in\dom(g^*)} \|\theta\vw\|_* \le \lambda$ is equivalent to $\|\vw\|_* \leq \lambda/M.$ Meanwhile, by introducing $u_i = b_i \va_i^T \vw, v_i = -u_i (i\in[n])$  and $s_i, t_i (i\in[n]$, we obtain Theorem \ref{thm:dro-wass}.
\end{proof}

In Theorem \ref{thm:dro-wass}, when we assume that the conic constraints $g(u) \le s$ and $\|\vw\|_*\le \lambda/M$ in Eq.~\eqref{eq:dro-wass} admit efficient proximal operators, we can formulate this problem as \eqref{eq:pd-glp} and apply \clvr. The resulting complexity bounds are similar to those discussed above for the $f$-divergence formulation. 




\section{Experiment details}\label{appx:experiments-details}




\subsection{Comparison of adaptive restart schemes}
\label{subsec-restart_schemes_comparisons}

We provide a brief empirical comparison between our adaptive restart scheme that uses LPMetric and the adaptive restart scheme using the normalized duality gap proposed in~\cite{applegate2021faster}. 
We compared the performance of PDHG on benchmark problem sets \texttt{qap10}, \texttt{qap15}, \texttt{nug08}, and \texttt{nug20} used in~\cite{applegate2021faster}, using the two adaptive restart criteria. 
We ran PDHG until reaching accuracy as described in~\cite{applegate2021faster} (that is, until normalized duality gap is at most $10^{-6}$ and primal and dual infeasibility is at most $10^{-8}$).

\begin{table*}[h!]
\centering
\begin{threeparttable}[b]
\begin{small}
\caption{Number of iterations required for the normalized duality gap and primal and dual infeasibility to fall below $10^{-6}$ and $10^{-8}$, respectively. }
\tabcolsep=0.1cm 
\begin{tabular}{|c|c|c|}
\hline
	Problem Name		&  	\text{Adaptive Normalized Duality Gap}	& \text{Adaptive LPMetric}  \\ \hline
	\texttt{qap10} & \bf{13041} & 14521 \\
	\texttt{qap15} & 12561 & \bf{961} \\ 
	\texttt{nug08} & \bf{841} & 1481  \\ 
	\texttt{nug20} & 22001 & \bf{16281} \\ \hline
\end{tabular} \label{table:restart}
\end{small}
\end{threeparttable}
\end{table*}

Table~\ref{table:restart} shows that the two restart criteria give similar performance in terms of iteration complexity.
Normalized gap is better on \texttt{qap10} and \texttt{nug08}, while LPMetric is better on \texttt{qap15} and \texttt{nug20}. 
For further details, Figure~\ref{fig:plots-restart-iter} plots the normalized duality gap vs iteration count.  
The two adaptive restart schemes lead to similar performance of PDHG over iterations. 
Comparisons based on wall-clock time are shown in Figure~\ref{fig:plots-restart-time}; the behavior is similar. 
We conclude that our restart criterion based on LPMetric seems comparable with normalized duality gap, in terms of iteration complexity.

\begin{figure*}[ht!]
\captionsetup[subfigure]{labelformat=empty}
    \centering
    \subfloat{{\includegraphics[width=0.45\textwidth]{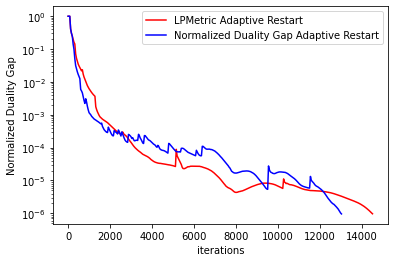} }\label{fig:qap10-iter}}%
    \hspace{\fill}
    \subfloat{{\includegraphics[width=0.45\textwidth]{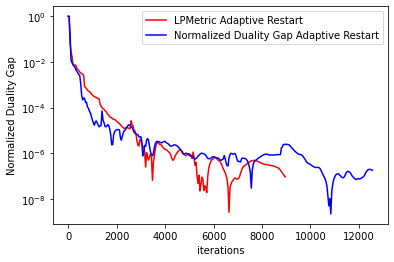} }\label{fig:qap15-iter}}%
    
    \subfloat{\includegraphics[width=0.45\textwidth]{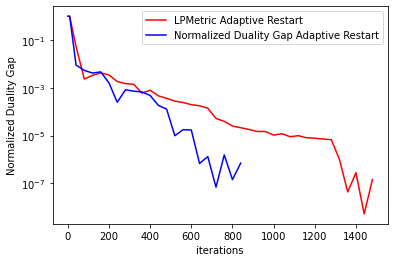}\label{fig:nug08-iter}}
    \hspace*{\fill}
    \subfloat{{\includegraphics[width=0.45\textwidth]{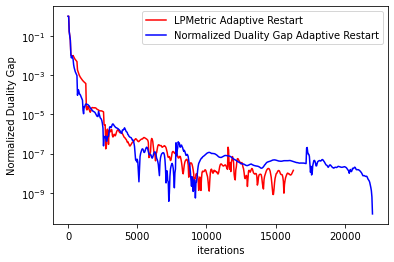}}\label{fig:nug20-iter}}
    \caption{Comparisons of restart schemes that use LPMetric and that use the normalized duality gap against number of iterations. The plots from left to right and then top to bottom are for \texttt{qap10}, \texttt{qap15}, \texttt{nug08}, and \texttt{nug20}.}
    \label{fig:plots-restart-iter}
\end{figure*}

\begin{figure*}[ht!]
\captionsetup[subfigure]{labelformat=empty}
    \centering
    \subfloat{{\includegraphics[width=0.45\textwidth]{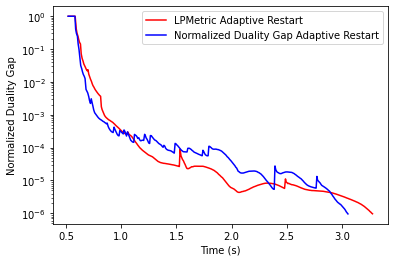} }\label{fig:qap10-time}}%
    \hspace{\fill}
    \subfloat{{\includegraphics[width=0.45\textwidth]{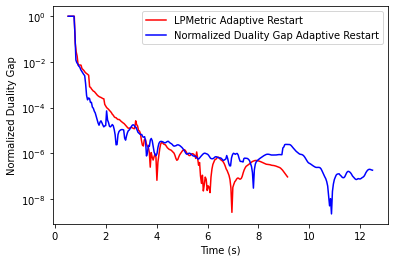} }\label{fig:qap15-time}}%
    
    \subfloat{\includegraphics[width=0.45\textwidth]{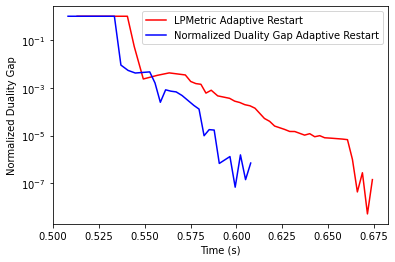}\label{fig:nug08-time}}
    \hspace*{\fill}
    \subfloat{{\includegraphics[width=0.45\textwidth]{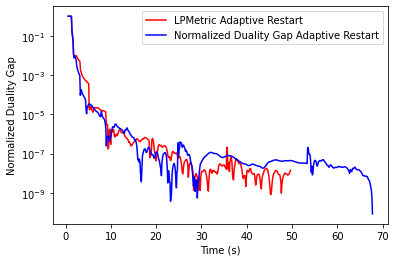}}\label{fig:nug20-time}}
    \caption{Comparisons of restart schemes that use LPMetric and that use the normalized duality gap against wall clock time in terms of seconds. The plots from left to right and then top to bottom are for \texttt{qap10}, \texttt{qap15}, \texttt{nug08}, and \texttt{nug20}.}
    \label{fig:plots-restart-time}
\end{figure*}




\subsection{Details of experiments in Section~\ref{sec:num-experiments-discussion}}

When we consider the DRO problem with Wasserstein metric of $\ell_1$-norm and with hinge loss, we observe that the reformulation described in Theorem~\ref{thm:dro-wass} can be further reformulated into an ordinary LP, as the dual norm of $\ell_1$ norm is $\ell_{\infty}$ norm and hinge loss can be decomposed linearly with additional auxiliary variables. Thus in the following instances we consider, we apply our adaptive restart scheme with respect to LPMetric as illustrated in Section~\ref{sec:restart} to achieve heuristic linear convergence rate in terms of the number of data passes.
We compare our \clvr~method with three representative methods: \textsc{pdhg}~\cite{chambolle2011first}, \textsc{spdhg}~\cite{chambolle2018stochastic} and \textsc{pure-cd}~\cite{alacaoglu2020random}.
We implemented \textsc{clvr} and other algorithms in \href{https://julialang.org}{Julia}, optimizing all implementations to the best of our ability.\footnote{Julia is particularly designed for high performance numerical computation.} 
For \textsc{spdhg}, whose per-iteration cost is at least $O(d)$, we consider a large batch size of $50$ to  balance the effect of the $O(d)$ cost and improve the overall efficiency. 
Meanwhile, \textsc{pure-cd} with block size $1$ is already well suited to sparsity. 
For \clvr, we experiment with block sizes $1$ and $10$. 

We conducted our experiments on \href{https://www.csie.ntu.edu.tw/~cjlin/libsvmtools/datasets/binary.html}{LibSVM}~\cite{chang2011libsvm} datasets \texttt{a9a},  \texttt{gisette}, and  \texttt{rcv1.binary}, each with different sparsity levels.
We run each algorithm using one CPU core, on a Linux machine with a second generation Intel Xeon Scalable Processor (Cascade Lake-SP) with $128$~GB of RAM. 
Because the weight parameter $\gamma$  between primal and dual variables (see Theorem~\ref{thm:clvr}) strongly influences empirical performance, we tune it for all datasets by trying the values $\{10^{-i}\}$ for $i\in\sZ$, for each of the methods. 
We set the Lipschitz constant  of \textsc{pdhg} to be the largest singular value of the constraint matrix in the LP formulation.
For \textsc{pure-cd} and \clvr~with block size $1$, because the rows of the matrix are normalized,  we set the Lipschitz constant to  $1$.
For \clvr~with block size $10$ and \textsc{spdhg} with block size $50$, the Lipschitz constants are tuned to $3$ and $9$, respectively.

\begin{remark}[Comparisons of using different block sizes]

We conducted experiments to compare the practical performance of \clvr~against different choices of block sizes, with results shown in Figure~\ref{fig:blocksize}. 
We ran the DRO with Wasserstein metric using the same setup as described in Section~\ref{sec:num-experiments-discussion}, on the \texttt{rcv1} dataset and using an early stopping criterion of \texttt{LPMetric} at $10^{-1}$.
In the plot, we can see that \clvr~converges to an approximate solution fastest when the block size is set to $10$, providing support for our choice of $10$ in Section~\ref{sec:num-experiments-discussion}. 
As illustrated in Figure~\ref{fig:plots}, \clvr~is most efficient in terms of the number of data passes when the block size is 1, but in terms of the execution time, running \clvr~with larger block size yields better performance. We attribute this phenomenon to the instruction-level parallelism~\cite{hennessy2011computer} in modern processors, allowing more computations to be completed in the same number of clock cycles. 
\end{remark}

\begin{figure*}[ht!]
    \captionsetup[subfigure]{labelformat=empty}
        \centering
        \includegraphics[width=0.60\textwidth]{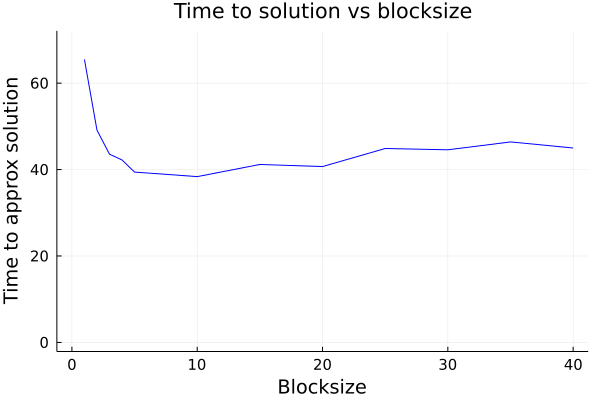}
        \caption{Comparison of the performance of CLVR with various choices of blocksizes.}
        \label{fig:blocksize}
\end{figure*}

In Table~\ref{tb:dataset}, we list information about the three datasets and the corresponding matrices in the reformulations. 
As we see, due to the sparse connectivity of auxiliary variables, all the matrices in reformulations are quite sparse. As a commonly adopted preprocessing step for LP, we normalize the matrix in the standard-form LP so that each row has Euclidean norm $1$.

\begin{table*}[ht]
\centering
\caption{The dimension and sparsity of the original datasets and the corresponding matrices in reformulations.}
\tabcolsep=0.1cm 
\begin{tabular}{|c|c|c||c|c|}
\hline 
Dataset	& \text{Original}\; $(d, n)$ 		& \#\text{nonzeros} / $(d \times n)$ 		& \text{Reformulated}\; $(d, n)$ 		& \#\text{nonzeros} / $(d \times n)$ \\ \hline
  \texttt{a9a} 	& $(123, 32561)$ 	& $0.11$ 	& $(130738, 97929)$	 &	$9.6\times10^{-5} $ \\ 
  \hline
  \texttt{gisette} 	& $(5000, 6000)$ 	& $0.99$ 	& $(44002, 28000)$ &	 $4.9\times10^{-2}$ \\
  \hline
  \texttt{rcv1} 	& $(47236, 20242)$ 	& $1.5\times10^{-3}$ 	& $(269914, 155198)$ &	 $8.8\times10^{-5}$ \\
  \hline
\end{tabular} \label{tb:dataset}
\end{table*}

\begin{remark}[Performance comparison using multiple cores]
We conducted further experiments to examine the effects of allowing the algorithms to run on more computing cores. However, we did not observe any meaningful difference in performances in terms of wall-clock time when we repeated the experiments described above and in Section~\ref{sec:num-experiments-discussion} using $2$ CPU cores per algorithm. Our interpretation of this observation is that because most of the steps within \clvr~and other algorithms we are comparing against are simple and cheap, involving very few large matrix-vector multiplications and no matrix factorization, the practical performance of algorithms becomes memory-bound, hence additional cores do not make much difference.
\end{remark}

\end{document}